\documentclass[12pt]{amsart}
\usepackage{amsaddr}
\usepackage[margin=1in]{geometry} 
\usepackage[utf8]{inputenc}
\usepackage{amsmath, graphicx}
\usepackage{color}
\usepackage{amsthm}
\usepackage{amssymb}
\usepackage{amscd}
\usepackage{mathtools}
\usepackage{hyperref}
\usepackage{tikz}
\usepackage{comment}
\usepackage{cleveref}
\usepackage{cancel}
\usepackage[shortlabels]{enumitem} 
\usepackage[numbers]{natbib}

\usetikzlibrary{patterns}
\tikzset{
  norm/.style     = {shape=circle, draw},
  blue/.style     = {shape=circle, draw, fill=blue!25},
  high/.style     = {shape=circle, draw, color=red},
  bluehigh/.style = {shape=circle, draw, color=red, fill=blue!25},
  red/.style      = {shape=circle, draw, fill=red!25},
  both/.style     = {shape=circle, draw, fill=violet!35},
  root/.style     = {node, bottom color=red!30},
  env/.style      = {treenode, font=\ttfamily\normalsize},
  dummy/.style    = {circle}
}
\tikzstyle{standard}=[circle, draw=black, fill=white, very thick, minimum size=7mm]
\tikzstyle{standard2}=[circle, draw=black, fill=white, very thick]
\tikzstyle{blue2}=[circle, draw=black, fill=blue!25, very thick]
\tikzstyle{small}=[circle, draw=black, fill=black, very thick, minimum size=4mm]
\tikzstyle{special}=[circle, draw=red!60, fill=red!5, very thick, minimum size=5mm]
\newtheorem{theorem}{Theorem}[section]
\newtheorem{lemma}[theorem]{Lemma}
\newtheorem{cor}[theorem]{Corollary}
\newtheorem{prop}[theorem]{Proposition}
\theoremstyle{definition}
\newtheorem{df}[theorem]{Definition}

\newtheorem{ex}[theorem]{Example}

\newtheorem{conj}[theorem]{Conjecture}

\DeclareMathOperator{\susp}{susp}

\DeclareMathOperator{\lk}{lk}
\DeclareMathOperator{\st}{st}
\DeclareMathOperator{\del}{del}
\DeclareMathOperator{\bbS}{\mathbb{S}} 

\title{Total Cut Complexes of Graphs}

\author[M. Bayer]{Margaret Bayer}
\address{University of Kansas, Lawrence, Kansas, USA, bayer@ku.edu}

\author[M. Denker]{Mark Denker}
\address{University of Kansas, Lawrence, Kansas, USA, mark.denker@ku.edu}

\author[M. Jeli\'c Milutinovi\'c]{Marija Jeli\'c Milutinovi\'c}
\address{University of Belgrade, Serbia, marijaj@matf.bg.ac.rs}

\author[R. Rowlands]{Rowan Rowlands}
\address{University of Washington, Seattle, Washington, USA, rowanr@uw.edu}

\author[S. Sundaram]{Sheila Sundaram}
\address{University of Minnesota, Minneapolis, Minnesota,  USA, shsund@umn.edu}

\author[L. Xue]{Lei Xue}
\address{University of Michigan, Ann Arbor, Michigan, USA, leixue@umich.edu}

\begin{document}    


\begin{abstract}
Inspired by work of Fr\"oberg (1990), and Eagon and Reiner (1998), we define the \emph{total $k$-cut complex} of a graph $G$ to be the simplicial complex whose facets are the complements of independent sets of size $k$ in $G$. We study the homotopy types and combinatorial properties of total cut complexes for various families of graphs, including chordal graphs, cycles, bipartite graphs, the prism $K_n \times K_2$, and grid graphs, using techniques from algebraic topology and discrete Morse theory.
\end{abstract}

\maketitle

\section{Introduction}
In recent years, there has been much interest in the topology of simplicial
complexes associated with graphs.  A comprehensive reference is Jonsson's book  \cite{JonssonBook2008}.

  A \emph{graph complex} is a simplicial complex associated to a finite graph $G$. In this paper we introduce a new family of graph complexes which we call \emph{total cut complexes}. Our work is motivated by a famous theorem of Ralf Fr\"oberg connecting commutative algebra and graph theory through topology.  We investigate our new complexes in the spirit of Fr\"oberg's theorem, relating topological properties to structural properties of the graph. 

For a field $\mathbb{K}$ and a finite simplicial complex $\Delta$ with vertex set $[n]$, the Stanley--Reisner ideal of $\Delta$ is the ideal $I_\Delta$  of the polynomial ring $\mathbb{K}[x_1,\ldots , x_n]$ generated by monomials $x_{i_1}\cdots x_{i_k}$, where 
$\{i_1, \ldots , i_k\}$ runs over the inclusion-minimal subsets of $[n]$ which are NOT faces of $\Delta$.  The Stanley--Reisner ring $\mathbb{K}[\Delta]$ is the quotient of the polynomial ring $\mathbb{K}[x_1,\ldots , x_n]$ by the ideal $I_\Delta$. 

For a graph $G$, 
the \emph{clique complex} $\Delta(G)$  is defined to be the simplicial complex whose simplices are subsets of vertices of $G$, in which every pair of vertices is connected by an edge of $G$.  Fr\"oberg \cite{Froberg1990} characterized ideals generated by monomials which have a \emph{$2$-linear resolution}, by first reducing to the case of square-free monomial ideals.    The ideal $I_\Delta$ is generated by quadratic square-free monomials precisely when the simplicial complex $\Delta$ is the clique complex $\Delta(G)$ for some graph $G$ (see \cite[Proposition~8]{EagonReiner1998}).  Hence Fr\"oberg's theorem can be stated as follows:

\begin{theorem}\label{thm:Fr}(Fr\"oberg) (\cite{Froberg1990}, \cite[p. 274]{EagonReiner1998}) A Stanley--Reisner ideal $I_\Delta$ generated by quadratic square-free monomials has a $2$-linear resolution if and only if $\Delta$ is the clique complex $\Delta(G)$ of a chordal graph $G$.
\end{theorem} 

Define the  \emph{combinatorial Alexander dual} of a simplicial complex $\Delta$ \cite[p.188]{BrunsHerzog1997} on $n$ vertices to be 
\[\Delta^{\vee} \coloneqq \{F\subset [n]: [n]\setminus F\notin \Delta\}.\]  
The $i$th homology of $\Delta$ and the $(n-i-3)$th cohomology of $\Delta^{\vee}$ are isomorphic by Alexander duality in the sphere $\bbS^{n-2}$.

For a graph $G$, write $\Delta_2(G)$ for 
 the  Alexander dual $\Delta(G)^\vee$ of the clique complex $\Delta(G)$.
Observe that the facets of  $\Delta_2(G)$  are the complements of independent sets of size 2 in $G$.


Eagon and Reiner's reformulation   of Fr\"oberg's theorem 
establishes the following equivalences.

\begin{theorem}\cite[Proposition~8]{EagonReiner1998} The following are equivalent for a graph $G$:
\begin{enumerate}
\item $\Delta_2(G)$ is vertex decomposable.
\item $\Delta_2(G)$ is Cohen-Macaulay over any field $k$.
\item $\Delta_2(G)$ is Cohen-Macaulay over some field $k$.
\item $G$ is chordal.
\end{enumerate}
\end{theorem}
To these equivalences one can add a fifth item, namely that  
$\Delta_2(G)$ is shellable, since shellability lies between  vertex decomposability and Cohen-Macaulayness.
(See, for example, 
\cite[Section 11]{BjTopMeth1995}.)

Inspired by this theorem, we introduce the following two  generalisations of the simplicial complex $\Delta_2(G)$.  Let $k\ge 1$.  Define 
\begin{enumerate}
\item a complex whose facets are complements of independent sets of size $k$ in $G$;  we call this the total $k$-cut complex of $G$, denoted $\Delta^t_k(G)$ (this paper);
\item  a complex whose facets are complements of sets  $F$ of size $k$ in $G$ such that the induced subgraph of $G$ on  the vertex set $F$ is disconnected; we call this the $k$-cut complex of $G$, and denote it by 
$\Delta_k(G)$ \cite{BDJRSX}.
\end{enumerate}

The subject of this paper is the first generalisation, the total cut complex $\Delta^t_k(G)$.  We examine its topology   and consider how it is affected by properties of the graph $G$.  In particular we add to the Eagon-Reiner equivalences of Fr\"oberg's theorem above by showing (see Theorem~\ref{thm:Lei-chordal-totk-vertdecomposable}):
\begin{theorem}
$\Delta_k^t(G)$ is vertex decomposable for \emph{all} $k$  $\iff$ $G$ is a chordal graph.
\end{theorem}

We also show that for many families of graphs,  the homotopy type of $\Delta^t_k(G)$ is of interest in its own right, often a wedge of spheres of equal dimension.

In the  companion \cite{BDJRSX} to this  paper we make a similar investigation of  the cut complex $\Delta_k(G)$.  

\section{Definitions}

All our graphs will be simple, that is, without loops 
or multiple
edges.  (Terminology follows \cite{WestGraphTheory1996}.)
\begin{df}
Let $G = (V,E)$ be a graph.

A set $S\subseteq V$ is an {\em independent set} if and only if 
no pair of vertices in $S$ forms an edge of~$G$. Note that every set of size one is independent.

The {\em independence number} $\alpha(G)$ of $G$ is the cardinality of 
a maximum independent set in~$G$.
\end{df}
\begin{df}\label{def:simplicial-complex}  A simplicial complex $\Delta$ on a set $A$ is a collection of subsets of $A$ such that 
\[\sigma\in \Delta \text{ and } \tau\subseteq \sigma \Rightarrow \tau \in \Delta.\] 
\end{df}
The elements of $\Delta$ are its \emph{faces} or \emph{simplices}. 
If the collection of subsets is empty, i.e., \ $\Delta$ has no faces,  we call $\Delta$ the void complex.  Otherwise $\Delta$ always contains the empty set as a face. 
The \emph{dimension} of a face $\sigma$, $\dim(\sigma)$, is one less than its cardinality; thus the dimension of the empty face is $(-1)$, and the 0-dimensional faces are the \emph{vertices} of $\Delta$. A \emph{$d$-face} or \emph{$d$-simplex} is a face of dimension $d$.
A {\em facet} of a simplicial complex is a maximal face. 

The \emph{join} of two simplicial complexes $\Delta_1$ and $\Delta_2$ with disjoint vertex sets is the complex \begin{center}{$\Delta_1 * \Delta_2= \{\sigma\cup \tau: \sigma\in \Delta_1, \tau\in \Delta_2\}.$}\end{center}

The \emph{cone} $\mathcal{C}_v(\Delta)$, with cone point $v$, over $\Delta$, and the \emph{suspension} of $\Delta$ are the complexes 
\begin{center}{$\mathcal{C}_v(\Delta)=\Delta* \Gamma_1, \ 
\mathrm{susp}(\Delta)=\Delta*\Gamma_2,$}\end{center}
where $\Gamma_1$ is the 0-dimensional simplicial complex $\{v\}$ with one vertex $v\notin \Delta$, and $\Gamma_2$ is the 0-dimensional complex with two vertices $u,v\notin \Delta$. 

The complexes we
will be considering are {\em pure}, meaning that all facets are of the same dimension, called the dimension of the simplicial complex. A {\em ridge} of a  pure simplicial complex is a face of dimension  one lower than the dimension of the complex.

The \emph{wedge} of two topological spaces $X$ and $Y$ with distinguished points $x_0\in X$, $y_0\in Y$, is the quotient of the disjoint union of $X$ and $Y$ obtained by identifying $x_0$ and $y_0$ to a single point.
As in \cite{Hatcher2002}, we write $X\bigvee Y$ for the {wedge} of the spaces $X,Y$, $\bbS^d$ for the $d$-dimensional sphere, and $\bigvee_m \bbS^d$ for a wedge of $m$ $d$-dimensional spheres. We use $\simeq$ to denote homotopy equivalence, and $\cong$ for homeomorphism of spaces and isomorphism of homology.

\begin{df}\label{def:tot-cut-cplx} Let $G$ be a graph and $k\ge 1$. 
The {\em total $k$-cut complex} $\Delta_k^t(G)$ is the simplicial complex whose 
facets are the complements of independent sets of size $k$.
(When $k$ is understood, we refer to the simplicial complex as the ``total cut complex''.) %
Equivalently, $\sigma$ is a face of $\Delta^t_k(G)$ if and only if the complement of $\sigma$ contains an independent set of size $k$.

An alternative viewpoint is that the facets of $\Delta_k^t(G)$ are the vertex covers of size $n-k$, i.e., sets of vertices that contain at least one endpoint of every edge.

Note that the vertices of $\Delta^t_k(G)$ are a (possibly proper) subset of the vertices of the graph $G$.
When $G$ has $n$ vertices and $k=1$,  the facets of the total cut complex are all the size $(n-1)$ subsets of the vertex set of $G$, and hence $\Delta_1^t(G)$ is simply the boundary of an $(n-1)$-simplex.  It therefore has the homotopy type of a single sphere in dimension $(n-2)$. 
\end{df}

In what follows we generally identify faces of the simplicial complex with
the corresponding vertex sets in the graph.

Note: if $k>\alpha(G)$, then $\Delta_k^t(G)$ is the void complex (the complex with no faces, not even the empty set), and if 
$G$ has $n$ vertices and 
$k\le \alpha(G)$, then $\Delta_k^t(G)$ has dimension $n-k-1$.  

In \cite{BDJRSX} we study the {\em $k$-cut complex} $\Delta_k(G)$
of a graph $G$.  
This is the simplicial complex whose facets are the vertex sets of size
$n-k$ whose removal leaves a disconnected subgraph. Note that for $k\ge 2$, $\Delta_k^t(G)$ is a subcomplex of the $k$-cut complex $\Delta_k(G)$.
Also, for $k=2$ the total cut complex $\Delta_2^t(G)$ coincides with the cut 
complex $\Delta_2(G)$.

The set of $k$-cut complexes  is much broader than the set of 
total $k$-cut complexes.  
In fact, it is shown in \cite{BDJRSX} that every pure simplicial complex is
a $k$-cut complex of some graph $G$ for some $k\ge 2$. This is not true for total cut complexes.
For example,
suppose the total cut complex $\Delta_3^t(G)$ of a graph $G$ with vertex set $\{1, \dotsc, 6\}$ has facets $\{1,2,3\}$, $\{1,3,4\}$ and $\{1,5,6\}$.
We claim that $\{1,3,6\}$ must be a facet as well. 
To see this, note that because $\{1,5,6\}$ is a facet, 
$24$ is not an edge of $G$;
because $\{1,3,4\}$ is a facet, $25$ is not an edge of $G$;
and because $\{1,2,3\}$ is a facet, $45$ is not an edge
of $G$.  So $\{2,4,5\}$ is an independent set of $G$, and $\{1,3,6\}$ is a facet of $\Delta_3^t(G)$. 

We are particularly interested in the topological properties of total cut
complexes.  These include shellability, vertex decomposability,
homotopy type and homology.  A good reference for definitions and basic
theorems is \cite{Koz2008}.

\begin{df}[{\cite[Chapter III, Section 2]{RPSCCA1996}}] \label{def:shelling} 
An ordering $F_1,F_2,\dots,F_t$ of the facets of a simplicial complex $\Delta$ 
is a \emph{shelling} if, for every $j$ with $1<j\leq t$,
$$\left( \bigcup_{i=1}^{j-1}\langle F_i\rangle\right) \cap \langle 
F_j\rangle$$
is a simplicial complex whose facets all have cardinality $|F_j|-1$, where 
$\langle F_i\rangle$ is the simplex generated by the face. 
If $\Delta$ has a shelling, $\Delta$ is called \emph{shellable}.
\end{df}

Note that the complex consisting only of the empty set is shellable, as are all 0-dimensional complexes.  To simplify some statements, we will use the convention that the void complex is shellable.

\begin{theorem}[{\cite[Theorem~1.3]{Bj-AIM1984}}] \label{thm:shell-implies-homotopytype}
A pure shellable simplicial complex of dimension $d$ is either contractible or it has the homotopy type of a wedge of spheres, all of dimension $d$.
\end{theorem}

\begin{ex}
Let $G$ be the graph in Figure~\ref{small-example}.
\begin{figure}
\begin{center}
\begin{tikzpicture}
\draw (2,1) node[below] {$d$}--
      (1,2) node[above] {$a$}--
      (0,1) node[left] {$b$}--
      (1,0) node[below] {$c$}--
      (2,1)--
      (3,1.5) node[right] {$e$};
\draw (2,1)--(3,0.5) node[right] {$f$};
\filldraw[black] (0,1) circle (2pt);
\filldraw[black] (1,0) circle (2pt);
\filldraw[black] (1,2) circle (2pt);
\filldraw[black] (2,1) circle (2pt);
\filldraw[black] (3,1.5) circle (2pt);
\filldraw[black] (3,0.5) circle (2pt);
\end{tikzpicture}
\end{center}
\caption{A graph $G$ with $\Delta_3^t(G)$ shellable and contractible.\label{small-example}}
\end{figure}
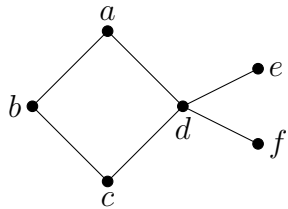

The independent sets of size 3 in $G$ are $\{b,e,f\}$ and all 3-element subsets
of $\{a,c,e,f\}$.  Thus the facets of $\Delta_3^t(G)$ are
$\{a,b,d\}$, $\{a,c,d\}$, $\{b,c,d\}$, $\{b,d,e\}$, and $\{b,d,f\}$.
This ordering of the facets is a shelling order.
The complex is contractible.

\end{ex}

We will be using the following constructions \cite[Section~2.1.2]{Koz2008}.
\begin{df}Let $\Delta$ be a simplicial complex and $\sigma$ a face of $\Delta$.
\begin{itemize}
\item The \emph{link} of $\sigma$ in $\Delta$ is 
    $$\lk_{\Delta} \sigma= \{\tau \in \Delta \mid \text{$\sigma\cap \tau 
      = \emptyset$, and $\sigma\cup \tau \in \Delta$}\}.$$
\item The (closed) \emph{star} of $\sigma$ in $\Delta$ is 
     $$\st_{\Delta} \sigma= \{\tau \in \Delta \mid  
        \sigma \cup \tau \in \Delta\} .$$
\item The \emph{deletion} of $\sigma$ in $\Delta$ is 
      \[\del_{\Delta} \sigma = \{\tau \in \Delta \mid \sigma \not\subseteq
       \tau\}.\]
\end{itemize}
\end{df}
(Note that in the deletion, we are not removing proper faces of $\sigma$.)

For $v$ a vertex and $\sigma$ any face of $\Delta$, we also have the following useful facts (see \cite{Koz2008}):
\begin{equation}\label{eqn:st-del-lk}\Delta=\st_\Delta(v)\cup \del_\Delta(v),\ \lk_\Delta(v)=\st_\Delta(v)\cap \del_\Delta(v) \text{ and } \st(\sigma)=\sigma*\lk(\sigma).\end{equation}
In particular $\st(\sigma)$ is contractible for any face $\sigma$.

\begin{df}\label{def:vertex-decomposable}
A $d$-dimensional simplicial complex $\Delta$ is called \emph{vertex decomposable} if either $\Delta$ is a simplex, or there is a vertex $v$ in $\Delta$ such that
    \begin{enumerate}[(1)]
        \item both $\lk_{\Delta} v$ and $\del_{\Delta} v$ are vertex decomposable, and 
        \item $\del_{\Delta} v$ is pure and $d$-dimensional.
    \end{enumerate}
\end{df}
Vertex decomposable simplicial complexes were introduced by Provan and Billera in \cite{Provan-Billera}. It is known that vertex decomposability implies shellability (see \cite[Corollary 2.9]{Provan-Billera}).
Shellability is preserved by the operation of taking links of 
faces; see \cite[Proposition~10.14]{BjWachsII1997}, 
\cite[Theorem~3.1.5]{WachsPosetTop2007}.

An important tool used in determining homotopy type of a simplicial complex is
discrete Morse theory.  We include an Appendix, which gives the definitions and results needed.

\section{General properties}

Note first that Definition~\ref{def:tot-cut-cplx}  %
implies the inclusion
\begin{equation} \label{eqn:inclusion}
\Delta^t_{k+1}(G)\subseteq \Delta^t_k(G).
\end{equation}
As in the case of regular cut complexes \cite{BDJRSX}, the faces of the total $(k+1)$-cut complex of a graph $G$ are completely determined by those of the total $k$-cut complex.  %

\begin{theorem}\label{theorem:facets, total cut}
    Let $k\geq 2$, $G=(V,E)$ a graph. Then the facets of $\Delta^t_{k+1}(G)$ are precisely the ridges of $\Delta^t_k(G)$ that are contained in exactly $k+1$ facets.
\end{theorem}
\begin{proof}
If $F=V\setminus S$ is a facet of $\Delta^t_{k+1}(G)$, then $S$ is an 
independent set of $G$ of size $k+1$, and so the $k+1$ subsets of $S$ of size
$k$ are all independent.  
Thus, $V\setminus S$ is a ridge of $\Delta^t_k(G)$ such that for any of the
$k+1$ vertices of $S$, $V\setminus (S\setminus x)$ is a facet of 
$\Delta^t_k(G)$.
Conversely, if $R$ is a ridge of $\Delta^t_k(G)$ contained in exactly
$k+1$ facets of $\Delta^t_k(G)$, then $S=V\setminus R$ contains $k+1$ 
elements, and the facets of $\Delta^t_k(G)$ containing $R$ must be exactly
$R\cup \{x\}$ for $x\in S$.  Thus $S\setminus x=V\setminus (R\cup\{x\})$ is
an independent set for each $x\in S$.  Since $k\ge 2$, every two of these
sets overlap, so the entire set $S$ is independent.  So $R=V\setminus S$ is
a facet of $\Delta^t_{k+1}(G)$.
\end{proof}

We consider how modifications of a graph affect the total cut complex.

\begin{prop}\label{prop:adding-isolated-vertex}
Let $G$ be a graph, $v\not\in V(G)$, and  $G \sqcup v$ the graph consisting of $G$ and the additional isolated vertex $v$. Then
$\Delta_k^t(G\sqcup v) = \Delta^t_{k-1}(G) \cup (\Delta_k^t(G)\ast v)$.
\end{prop}
\begin{proof}
The facets of $\Delta^t_k(G\sqcup v)$ are complements of independent sets of
size $k$ in $G\sqcup v$.  There are two types of such independent sets $S$.

If $v\in S$, then $S\setminus v$ is an independent set of size $k-1$ in
$G$.
Since the complement of $S$ in $G\sqcup v$ equals the complement of $S\setminus v$
in $G$, these facets generate a subcomplex of $\Delta_k(G\sqcup v)$ equal to
$\Delta^t_{k-1}(G)$.

If $v\not\in S$, then $S$ is an independent set of $G$ of size $k$.
The complement of $S$ in $G\sqcup v$ is the union of $\{v\}$ with the
complement of $S$ in $G$.  So these facets generate a subcomplex of
$\Delta_k(G\sqcup v)$ equal to $\Delta^t_k(G)\ast v$.
\end{proof}

\begin{lemma}\label{lem:link of faces, total cut}
Let $k\geq 2$, $G=(V,E)$ be a graph, and $W$ be a face of $\Delta^t_k(G)$. 
Then $\Delta^t_k (G\setminus W) = \lk_{\Delta^t_k(G)} W$.
\end{lemma}
\begin{proof}
For $F\subseteq V\setminus W$, $F$ is a facet of $\Delta^t_k(G\setminus W)$
if and only if $(V\setminus W)\setminus F=V\setminus (W\cup F)$ is an 
independent set in $G\setminus W$ (and hence of $G$) of size $k$
if and only if $W\cup F$ is a facet of $\Delta^t_k(G)$ if and only if  
$F$ is a facet of $\lk_{\Delta^t_k(G)} W$.
\end{proof}

Note that if $W$ is not a face of $\Delta_k^t$, then $G\setminus W$ contains no set of $k$ independent vertices.  In that case $\Delta_k^t(G\setminus W)$ is the void complex.  Since links of shellable complexes are shellable, we have the following.
\begin{cor}\label{cor:induced subgraphs, total cut}
      Let $k\geq 2$, $G=(V,E)$ a graph, and $W\subseteq V$. If the total $k$-cut complex $\Delta^t_k(G)$ is shellable, so is $\Delta^t_k(G\setminus W)$.
\end{cor}
In other words, if the total cut complex of a graph is shellable, then so
is the total cut complex of every induced subgraph.

For $v$ a vertex of a graph $G$, the {\em neighborhood} of $v$ is the set
$N(v) = \{u\in V(G) \mid uv\in E(G)\}$.
A vertex $v$ is a {\em simplicial vertex} of $G$ if and only if $N(v)$ is a 
clique in $G$. %

\begin{lemma}\label{lem:pure subcomplex avoiding one vertex, total cut}
      Let $k\geq 2$, $G = (V,E)$ a graph, $v$ a simplicial vertex of $G$, $N(v)$ its neighborhood in $G$,  and $\Delta = \Delta^t_k(G)$. %
      Then $\del_{\Delta} v$ is a pure simplicial complex generated by the facets of $\Delta^t_{k-1}(G\setminus v)$ that contain $N(v)$, i.e.,
      \begin{equation*}%
          \del_{\Delta} v = \st_ {\Delta^t_{k-1}(G\setminus v)} N(v). 
      \end{equation*}
\end{lemma}

\begin{proof}
We start by showing that $\del_{\Delta} v$ is pure. For any $\tau\in \del_{\Delta} v$, we will show that there exists a facet of $\Delta$ that contains $\tau$ and not $v$. Since $\Delta$ is pure, there exists a facet $F\in \Delta$ such that $\tau \subseteq F$. If $v\not\in F$, we are done. Now assume $v \in F$, and
$F=V\setminus S$.  Since $S$ is an independent set and $N(v)$ forms a clique, there can be at most one neighbor of $v$ that is in $S$.  If such a neighbor of $v$ exists, call it $w$; otherwise pick any vertex in $S$ as $w$.  Then $S+v-w$ is independent, so $F' = F - v +w$ is another facet of $\Delta$.
Since $\tau\subseteq F'$ and $v\not\in F'$, $\tau \subseteq F' \in \del_{\Delta}v$. Therefore $\del_{\Delta}v$ is pure.

We now know that $\del_{\Delta}v$ is the pure simplicial complex generated by the facets of $\Delta$ that do not contain $v$. Let $F \in \del_{\Delta}v$ be one of these facets. Then $v\not\in F$, so the entire set $N(v)$ must be in $F$ and also $F \in \Delta^t_{k-1}(G\setminus v)$. This means $F$ is in the star of $N(v)$ in $\Delta_{k-1}^t(G\setminus v)$.
\end{proof}

\begin{theorem}\label{thm: tot-cut complex add simplicial vertex -Lei-M-S}
    Let $G$ be a graph, $v\in V(G)$ a simplicial vertex with nonempty neighborhood $N(v)$, and $k\geq 2$. Let $H=G\setminus v$.
    If $\Delta^t_k(H)$ is not the void complex, then there is a homotopy equivalence 
    \begin{equation}\label{eqn:susp-remove-simplicial-vertexS}
\Delta_k^t(G)\simeq \susp \Delta_k^t(G\setminus v),
\end{equation}
    and hence an isomorphism in homology $\Tilde{H}_n(\Delta_k^t(G))\cong\Tilde{H}_{n-1}(\Delta_k^t(H))$ for all $n\geq 1$. 
    
    If $\Delta^t_k(H)$ is the void complex, then $\Delta_k^t(G)$ is either contractible or the void complex. %

\end{theorem}

The proof of this theorem relies on the following general topological facts as well as two propositions. 

\begin{prop}[{\cite[Chapter 21, (21.3)]{VINK2008}}] \label{prop:TopFact2ndIsoThm} Let $A,B \subset X$ be such that $X=A\cup B$,  $A \cap B \neq \emptyset$, and $A, B$ are both closed subspaces, or both open subspaces,  of the topological space $X$.  Then the quotient map $A/(A\cap B) \rightarrow X/B$ of the inclusion $A\hookrightarrow X$ is a homeomorphism.
\end{prop}

\begin{proof}
Let $\phi : A \to A/(A \cap B)$ be the quotient map, and let $\psi : A \to X/B$ be the composition of the inclusion $A \hookrightarrow X$ and the quotient map $X \to X/B$. The map $\psi$ is surjective, since the point representing $B$ is the image of any point in $A \cap B$, and every other point of $X/B$ is the image of a point of $X$ outside $B$, and thus inside $A$. Also, if $A$ and $B$ are open sets, then $\psi$ is an open map: the inclusion $A \hookrightarrow X$ is open, and every open set $U$ of $X$ is either disjoint from $B$, hence its image under the quotient map $X \to X/B$ is unchanged, or it intersects $B$, hence its image under $X \to X/B$ is the same as the image of $U \cup B$, which is open. Similarly, if $A$ and $B$ are closed sets, then $\psi$ is a closed map, by the same argument with ``open'' replaced with ``closed''. Therefore, $\psi$ is a quotient map, by \cite[Proposition~3.67]{LeeITM2011}. However, both $\phi$ and $\psi$ make the same identifications: two distinct points of $A$ are identified under each map if and only if they are both in $A \cap B$. Therefore, by the uniqueness of quotient spaces \cite[Theorem~3.75]{LeeITM2011}, the induced map $A/(A \cap B) \to X/B$ is a homeomorphism.
\end{proof}

\begin{prop}[{\cite[Proposition~0.17, Ex.~0.14]{Hatcher2002}, \cite[Proposition~4.1.5]{Matousek2003}}]\label{prop:quotient-by-contractible-homotopy} If $(X,A)$ is a CW pair consisting of a CW complex $X$ and a contractible subcomplex $A$, then the quotient map $X\rightarrow X/A$ is a homotopy equivalence.

More generally, if the subcomplex $A$ is contractible in the complex $X$, then the quotient $X/A$ is homotopy equivalent to $X\bigvee\susp A$.

\end{prop}
\begin{prop} \label{prop tot-cut complex add edge -Lei2}
    Let $G$ be a graph, $v\in V(G)$ a simplicial vertex with neighborhood $N(v)$, and $k\geq 2$. Let $H=G\setminus v$. Then
    \begin{equation*}
        \Delta^t_k (G) = \big(\Delta^t_k(H)* v  \big) \cup \st_{\Delta^t_{k-1}(H)} N(v) .
    \end{equation*}
\end{prop}

\begin{proof} Apply the standard facts about the star and deletion of a vertex listed in ~\eqref{eqn:st-del-lk} to the simplicial complex $\Delta=\Delta_k^t(G)$ to obtain 
\[ \Delta=\st_{\Delta}(v)\cup \del_{\Delta}(v)
=\left(\lk_\Delta(v)*v\right) \cup \del_{\Delta}(v).\] 
 The result now follows from Lemma~\ref{lem:link of faces, total cut} and Lemma~\ref{lem:pure subcomplex avoiding one vertex, total cut}.\end{proof}

\begin{prop}\label{prop: tot-cut complex in Stars  -Lei2}
    Let $H$ be a graph and $k\geq 2$.  Assume $\Delta_k^t(H)$ is not the void complex. Then for an arbitrary clique $N\subseteq V(H)$,
    \begin{equation*}
        \Delta^t_k (H) \subseteq \st_{\Delta^t_{k-1}(H)} N.
    \end{equation*}
\end{prop}

\begin{proof}
Let $F$ be a facet of $\Delta^t_k (H)$. If $N\subseteq F$, then clearly $F\in \st_{\Delta^t_{k-1}(H)} N$ since, from~\eqref{eqn:inclusion}, $\Delta^t_k(H) \subseteq \Delta^t_{k-1}(H)$. If $N\nsubseteq F$, then $N\cap (V(H)\setminus F)$ must consist of a single element $u_0$, since two elements of the clique $N$ cannot be in an independent set. %
Then $N\subseteq F\cup u_0 \in \Delta^t_{k-1}(H)$, so $F\cup N\in \Delta^t_{k-1}(H)$, which also implies $F \in \st_{\Delta^t_{k-1}(H)} N$.
\end{proof}
\begin{proof}[Proof of Theorem~\ref{thm: tot-cut complex add simplicial vertex -Lei-M-S}]
Proposition \ref{prop tot-cut complex add edge -Lei2} says that 
\begin{equation*}
    \Delta^t_k (G) = {\underbrace{\Delta^t_k(H)* v}_{A}} \; \cup \; {\underbrace{\st_{\Delta^t_{k-1}(H)} N(v) }_{B}}.
\end{equation*}

 Note that if $B$ is not the void complex, then it is contractible because it is the star of a simplex, see ~\eqref{eqn:st-del-lk}.
We now have two cases. 

First suppose $\Delta_k^t(H)$ is the void complex. Then $A$ is the void complex, so $\Delta_k^t(G)=B=\st_{\Delta_{k-1}^t(H)} N(v)$. If $B$ is also the void complex, then $\Delta_k^t(G)$ is the void complex. If $B$ is not the void complex,  $B$, which is $\Delta_k^t(G)$, is contractible.

Now suppose $\Delta_k^t(H)$ is not the void complex, so $A$ is nonempty,  and from Proposition~\ref{prop: tot-cut complex in Stars  -Lei2},  $B$ is nonvoid and contractible.  

Moreover, $A\cap B=\Delta_k^t(H)$, since by Proposition \ref{prop: tot-cut complex in Stars  -Lei2}, $\Delta^t_k(H) \subseteq B$, so $\Delta_k^t(H)\subseteq A\cap B$, and $v\notin B$, so $A\cap B\subseteq \del_A v=\Delta_k^t(H)$. 

Write $A_1=\Delta^t_k(H)$ so that $A=\Delta^t_k(H)* v=\mathcal{C}_v(A_1)$ is the cone over $A_1$ with cone point $v$. %

Hence we have $X=\Delta^t_k (G)=\mathcal{C}_v(A_1)\cup B.$

Clearly $A$ and $B$ are closed subspaces of $X$, and $A\cap B=A_1\cap B=A_1$ since $v\notin B$.
 Thus  Proposition~\ref{prop:TopFact2ndIsoThm} applies, giving a homeomorphism
\begin{equation*}\label{eqn:homeo1}  X/B=(\mathcal{C}_v(A_1)\cup B)/B\cong \mathcal{C}_v(A_1)/A_1. \end{equation*}

Since $\mathcal{C}_v(A_1)/A_1$ is homeomorphic to $\susp(A_1)$,  we have the homeomorphism 
\begin{equation*}\label{eqn:homeo2} X/B\cong \susp(A_1).\end{equation*}

But $B$ is contractible, so by Proposition~\ref{prop:quotient-by-contractible-homotopy}, 
$X/B$ is homotopy equivalent to $X=\Delta^t_k (G)$.

We have established the homotopy equivalence 
$\Delta_k^t(G)\simeq \susp \Delta_k^t(G\setminus v)$
for a simplicial vertex $v$, as claimed.
 \end{proof}

A simple application of Theorem~\ref{thm: tot-cut complex add simplicial vertex -Lei-M-S} is provided by the path graph $P_n$. 
It is easy to see that   $\Delta_k^t(P_n)$ is nonvoid only if $n\ge 2k-1$. Also $\Delta_k^t(P_{2k-1})$ has exactly one facet and is thus contractible.
Since $P_n$ may be obtained by successively adding $(n-2k+1)$  simplicial vertices to the path $P_{2k-1}$, we can apply Theorem~\ref{thm: tot-cut complex add simplicial vertex -Lei-M-S} to conclude that $\Delta^t_k(P_n)$ is contractible for all $n \ge 2k-1$, since 
\[\Delta^t_k(P_n)\simeq \underbrace{\susp\cdots\susp}_{n-2k+1} \Delta^t_k(P_{2k-1})\simeq \text{ a point}.\]

A more interesting application of Theorem~\ref{thm: tot-cut complex add simplicial vertex -Lei-M-S} to a specific family of graphs is provided in the next section, in Example~\ref{ex:simplicial-vertexS}.

The homotopy equivalence~\eqref{eqn:add-isolated-vertexS} of Theorem~\ref{thm:G-with-isolated-vertex-MarijaS} below is a counterpart to~\eqref{eqn:susp-remove-simplicial-vertexS} in Theorem~\ref{thm: tot-cut complex add simplicial vertex -Lei-M-S}.
\begin{theorem}\label{thm:G-with-isolated-vertex-MarijaS}
    Let $G$ be a graph and $G \sqcup v$ the graph consisting of $G$ and an additional isolated vertex $v$.
    \begin{enumerate}
    \item[(a)] There is a homotopy equivalence 
    \begin{equation}\label{eqn:add-isolated-vertexS}\Delta_k^t(G\sqcup v)\simeq \Delta_{k-1}^t(G)/\Delta_k^t(G).
\end{equation}
If in addition $\Delta_k^t(G)$ is contractible in $\Delta_{k-1}^t(G)$, we have 
\begin{equation}\label{eqn:add-isolated-vertex-when-contactibleS}\Delta_k^t(G\sqcup v)\simeq \Delta_{k-1}^t(G)\bigvee \susp\Delta_k^t(G).
\end{equation}
    \item[(b)]
    For every $k \ge 2$ there is a long exact sequence in homology: 
$$\cdots \rightarrow \tilde{H}_i(\Delta^t_k(G)) \rightarrow  \tilde{H}_i(\Delta^t_{k-1}(G))
\rightarrow  \tilde{H}_i(\Delta^t_k(G \sqcup v))
\rightarrow  \tilde{H}_{i-1}(\Delta^t_k(G)) \rightarrow \cdots$$
\end{enumerate}
    \end{theorem}
\begin{proof} Part (a) follows as in the preceding proof. From Proposition~\ref{prop:adding-isolated-vertex} we have, since $v$ is an isolated vertex, 
\[\Delta_k^t(G\sqcup v)=\Delta_{k-1}^t(G)\cup \mathcal{C}_v(\Delta_k^t(G)), \]  
where as before, $\mathcal{C}_v$ denotes the cone with cone point $v$.
Hence by Proposition~\ref{prop:TopFact2ndIsoThm} and Proposition~\ref{prop:quotient-by-contractible-homotopy}, we have the homotopy equivalences 
\begin{equation*} 
\begin{split}\Delta_k^t(G\sqcup v) \simeq [\Delta_{k-1}^t(G)\cup \mathcal{C}_v(\Delta_k^t(G))]/\mathcal{C}_v(\Delta_k^t(G))\\
\simeq \Delta^t_{k-1}(G)/[\Delta_{k-1}^t(G)\cap\mathcal{C}_v(\Delta_k^t(G))  ].
\end{split}
\end{equation*}
    But the intersection $\Delta_{k-1}^t(G)\cap\mathcal{C}_v(\Delta_k^t(G))$ is precisely 
    $\Delta^t_k(G)$, 
    since, from~\eqref{eqn:inclusion}, $\Delta^t_k(G)\subseteq \Delta_{k-1}^t(G)$, and  $v\notin G$.
    The result follows from Proposition~\ref{prop:quotient-by-contractible-homotopy}.

    For Part (b), note that (e.g., \cite{Hatcher2002}) for a CW-complex $X$ and a subcomplex $A$, the reduced homology of the quotient complex $X/A$ is isomorphic to the relative homology of the pair $(X,A)$.  Hence the homotopy equivalence of Part (a) gives us the  homology isomorphism 
    \[\tilde{H}_i(\Delta^t_k(G\sqcup v))\cong H_i(\Delta^t_{k-1}(G), \Delta^t_k(G)).\] 
    The conclusion is now  immediate from the long exact homology sequence of the pair \cite{Hatcher2002}
    $(\Delta^t_{k-1}(G), \Delta^t_k(G))$.
 \end{proof}

\section{Specific classes of graphs}

We consider the total cut complexes of some specific classes of graphs.

\begin{prop}  Let $k\ge 2$.
\begin{itemize}
\item The total cut complex of the complete graph $K_n$ is the void complex.
\item For $k\le n$ the total cut complex $\Delta_k^t$ of the edgeless graph $\overline{K}_n$ is
      the $(n-k-1)$-skeleton of the $(n-1)$-simplex.  Therefore, 
      $\Delta^t_k(\overline{K}_n)$ is shellable and is homotopy equivalent to
      a wedge of $\binom{n-1}{k-1}$ spheres of dimension $n-k-1$. 
\end{itemize}
\end{prop}

We can now give an example to illustrate the homotopy equivalence of Theorem~\ref{thm:G-with-isolated-vertex-MarijaS}.
\begin{ex}\label{ex:G-with-isolated-vertex-MarijaS}  Let $G=\bar{K}_n$. Then $\bar{K}_{n+1}$ is obtained by adding an isolated vertex $v$ to $G$, i.e., $\bar{K}_{n+1}=G\sqcup v$.

Using the homotopy equivalences~\eqref{eqn:add-isolated-vertexS} and ~\eqref{eqn:add-isolated-vertex-when-contactibleS}, we have:
\[ \Delta_k(G\sqcup v)\simeq \Delta_{k-1}^t(G)/\Delta_k^t(G)  \simeq \Delta_{k-1}^t(G)\bigvee \susp(\Delta_k^t(G)), \]
where the second homotopy  follows from 
the fact that  the $m$-skeleton of a simplex is contractible in the $(m+1)$-skeleton.

Since $\Delta_{k-1}^t(G)\simeq \bigvee_{\binom{n-1}{k-2}} \bbS^{n-k}$ and $\Delta_k^t(G)\simeq \bigvee_{\binom{n-1}{k-1}} \bbS^{n-k-1},$ the right-hand side is homotopy equivalent to $\bigvee_{\binom{n}{k-1}} \bbS^{n-k}$,
which is indeed the homotopy type of  $\Delta_k(G\sqcup v)=\Delta_k(\bar{K}_{n+1})$.  

\end{ex}

For complete bipartite graphs and complete multipartite graphs,
the total cut complex $\Delta_k^t(G)$ coincides with the ordinary cut complex 
$\Delta_k(G)$: maximal disconnecting sets are complements of independent sets.
These complexes are studied in  more detail in the forthcoming paper \cite{BDJRSX}.  We present here the
basic result on complete bipartite graphs, stated for the total cut 
complex.
\begin{theorem}\label{thm:bipartite}(See also \cite{BDJRSX})
Let $1\le m\le n$ and $k\ge 2$.  Then
$\Delta_k^t(K_{m,n})$ is shellable if and only if  $m<k$.
Furthermore, if $m<k\le n$, then $\Delta_k^t(K_{m,n})$ is contractible, and if
$m=k=n$, then $\Delta_k^t(K_{m,n})$ is homotopy equivalent to the 0-sphere.
If $k>n$, the (total) cut complex is void and hence shellable.
\end{theorem}

We prove this theorem with several propositions covering different cases.

Let the two partite sets of $K_{m,n}$ be $V_m$ of size $m$
and $V_n$ of size $n$; write $V=V_m\sqcup V_n$.

 \begin{prop}\mbox{ }\label{prop:Bipartite:k=n}
 \begin{itemize}
 \item If $m<k=n$, then $\Delta^t_k(K_{m,n})$ is the $(m-1)$-simplex, so is trivially
       shellable and contractible.
 \item If $m=k=n$, then $\Delta^t_k(K_{m,n})$ is the union of two disjoint $(m-1)$-simplices, hence not shellable.  It is 
         homotopy equivalent to the 0-sphere. 
 \end{itemize}
 \end{prop}

 We turn now to the case where $m<k<n$.  %
 
 Here the $k$-element independent sets are the $k$-subsets of $V_n$, so the facets
of $\Delta_k^t$ are the $(m+n-k)$-element subsets of $V$ of the form
  $V_m\sqcup T$ where $T$ is a subset of $V_n$ of size $(n-k)$.
We show that the complex is then shellable as well as contractible.

 \begin{prop}\label{prop:Bipartite:m<k<n}  Let $m<k<n.$ Then $\Delta^t_k(K_{m,n})$ is shellable and contractible.
 \end{prop}

 \begin{proof}%
 This cut complex is the join of an $(m-1)$-simplex with the $(n-k-1)$-skeleton of an $(n-1)$-simplex.
 But a simplex is shellable, so its skeleton is  shellable \cite[Theorem 3.1.7]{WachsPosetTop2007}, and the join of two shellable complexes is also shellable \cite[Theorem 3.1.6]{WachsPosetTop2007}. See also \cite[Corollary 4.4, Theorem 8.1]{BjWachsTAMS1983}. Hence $\Delta^t_k(K_{m,n})$ is shellable.

 Since $\Delta^t_k(K_{m,n})$ is the join of a contractible simplicial complex, the $(m-1)$-simplex, with another complex, it must be contractible.
  \end{proof}

 We determine the homotopy type of  the total cut complex $\Delta^t_k(K_{m,n})$ for $k\le m\le n$ by a discrete Morse theory argument. 

 \begin{prop}\label{prop:MorseMatchingBipartite:$k<=m$}
 Let $K_{m,n}$ be a complete  bipartite graph  and $k$ such that $2\le k \le m\le n.$   Then  the total cut complex $\Delta^t_k(K_{m,n})$ has the following homotopy type: 
 \begin{equation}\label{k_less_m}
 \Delta^t_k(K_{m,n}) \simeq \bigvee_{\binom{m-1}{k-1} \cdot \binom{n-1}{k-1}} \bbS^{m+n-2k}.
 \end{equation}
 In particular when $k=m$ we have 
 \begin{equation}\label{k_eq_m}
 \Delta^t_m(K_{m,n}) \simeq \bigvee_{\binom{n-1} {m-1}} \bbS^{n-m}.
 \end{equation}
 \end{prop}

 \begin{proof}  A set of vertices of the complete bipartite graph $K_{m,n}$ is independent if and only of it is  contained in one of the two sets $V_m$ or $V_n$. So the complement $\sigma^c$ of a face $\sigma \in \Delta^t_k(K_{m,n})$ has to contain at least $k$ vertices in the same set $V_m$ or $V_n$.
 Label the vertices of $G$ as $V_m = \{a_1, \dots, a_m\}$ and $V_n=\{b_1, \dots, b_n\}$.
 We construct an element matching $\mathcal{M}_1$ using the vertex $a_1$, followed by an element matching $\mathcal{M}_2$ using the vertex $b_1$.

 After the matching $\mathcal{M}_1$, the unmatched  faces $\sigma$ are those faces  of $\Delta^t_k(K_{m,n})$ that satisfy:
 $$a_1 \notin \sigma, \ \  \sigma \in \Delta^t_k(K_{m,n}), \ \ \sigma \cup \{a_1\} \notin \Delta^t_k(K_{m,n}).$$
 This means that  $\sigma$ contains exactly $m-k\ge 0$ elements of the set $V_m$, and at least $n-(k-1)$ elements of the set
 $V_n$.  We continue with the element matching $\mathcal{M}_2$ using the vertex $b_1$, and consider the faces $\sigma$ that remain unmatched. Any face $\sigma$ not matched in $\mathcal{M}_1$,  which does not contain vertex $b_1$,  is matched with a face $\sigma \cup \{b_1\}$ in $\mathcal{M}_2$. Faces that remain unmatched after $\mathcal{M}_2$ are exactly those faces $\sigma$ that contain vertex $b_1$, but cannot be matched with $\sigma \setminus \{b_1\}$ because $\sigma \setminus \{b_1\}$ was already matched in $\mathcal{M}_1$.
 This means that $\sigma$ contains exactly $n-(k-1)$ elements of the set $V_n$. We conclude that the faces that are unmatched after $\mathcal{M}_1 \cup \mathcal{M}_2$ are those faces $\sigma \in \Delta^t_k(K_{m,n})$ that satisfy:
 $$a_1 \notin \sigma, \ \  |\sigma \cap V_m| = m-k\ge 0, \ \
 b_1 \in \sigma, \ \ |\sigma \cap V_n| = n-(k-1).$$
 A straightforward calculation implies that there are $\binom{m-1}{k-1} \cdot \binom{n-1}{k-1}$ unmatched faces, and each of them contains exactly $m+n-2k +1$ vertices.  Since $\mathcal{M}_1 \cup \mathcal{M}_2$ is an acyclic matching (Theorem~\ref{thm:seq-of-elt-match}),  the complex $\Delta^t_k(K_{m,n})$ is homotopy equivalent to a CW complex with one $0$-cell (matched to the empty set) and  $\binom{m-1}{k-1} \cdot \binom{n-1}{k-1}$ cells of dimension $m+n-2k$. This implies the desired homotopy type \eqref{k_less_m}, and the special case  \eqref{k_eq_m} when $k=m$.  \end{proof}

 Theorem~\ref{thm:bipartite} now follows from
Proposition~\ref{prop:Bipartite:k=n},  Proposition~\ref{prop:Bipartite:m<k<n} and Proposition~\ref{prop:MorseMatchingBipartite:$k<=m$}.


Next we consider the prism over a clique.
\begin{df}The {\em prism} $G_n$ over the complete graph $K_n$ is the Cartesian product of complete graphs $K_n\times K_2$. 
It has $2n$ vertices which we denote by $\{i^+, i^-: 1\le i\le n\}$, and edges 
$\{i^+j^+\}_{1\le i<j\le n}$, $\{i^- j^-\}_{1\le i<j\le n}$, $\{i^+ i^-\}_{1\le i\le n}$.
\end{df}
The independence number of $G_n$ is clearly 2, and hence the total cut complex is nonvoid only for $k=1,2$. We settle the case $\Delta^t_2(G_n), n\ge 2,$ again with a discrete Morse matching.

\begin{theorem}\label{thm:PrismCliquek=2My2ndDMM2022June21} The $(2n-3)$-dimensional (total) cut complex $\Delta^t_2(G_n), n\ge 2,$ has the homotopy type of a wedge of $(n-1)$ spheres $\bbS^{2n-4}$ of dimension $2n-4$.
\end{theorem}

\begin{proof}  We construct a Morse matching on the faces of the total cut complex and show that there are $(n-1)$ critical cells of dimension $2n-4$.   For example, the critical cells in $\Delta_2(G_3)$ will turn out to be $\{1^-, 2^+, 2^-\}, \{1^-, 3^+, 3^-\},$ while those for $\Delta_2(G_4)$ 
will be  $\{1^-, 2^+, 2^-, 3^+, 3^-\}, \{1^-, 2^+, 2^-, 4^+, 4^-\}$ and $ \{1^-, 3^+, 3^-, 4^+, 4\}.$ 

For each ordered pair $(i,j), 1\le i\ne j\le n,$ there is one facet  $F_{(i^+,j^-)^c}$ whose complement is the set  $\{i^+, j^-\}.$   In particular a subset $\sigma $ is a face of $\Delta^t_2(G_n)$ if and only if  $\sigma\subseteq F_{(i^+,j^-)^c}$ for some pair $(i,j), i\ne j$.  Hence  $\sigma$ is NOT a face if and only if 
\begin{equation*}\label{eqn:nonface-prismDelta2-0} 
i^+\notin \sigma \text{ for some } i \Rightarrow  j^-\in \sigma \text{ for all } j\ne i, \text{ and } 
 j^-\notin \sigma \text{ for some } j \Rightarrow i^+\in \sigma \text{ for all } i\ne j.
 \end{equation*}
 Let $V^+=\{1^+,2^+,\ldots,n^+ \}$ and $V^-=\{1^-,2^-,\ldots,n^- \}$. Equivalently, 
     $\sigma$ is not a face if and only if 
     \begin{equation}\label{eqn:nonface-prismDelta2}
\text{ either } V^+ \subseteq \sigma, \text{ or } V^- \subseteq \sigma, 
 \text{ or } 
 \sigma =V^+\cup V^- \setminus\{j^+,j^-\} \text{ for some }j.
\end{equation}
We use an element matching $\mathcal{M}_{1^+}$ first with the vertex $1^+$, followed by  the element matching $\mathcal{M}_{1^-}$ with vertex $1^-$. 
The degenerate case $n=2$ is easily seen to produce one unmatched 0-dimensional critical cell $\{1^-\}$ after these two element matchings, confirming the homotopy type in this case. The pairs in the matching are $\{(\emptyset, 1^+), (2^-,1^+2^-), (2^+,1^-2^+)\}$, leaving the 0-dimensional critical cell $\{1^-\}$. Note that for $n\ge 3$, $\{1^+,1^-\}$ is a face of $\mathcal{M}$, making the situation for $n=2$ special.

After the matching $\mathcal{M}_{1^+}$, the unmatched faces are $\sigma$ such that $1^+\notin \sigma$ and $1^+\cup\sigma$ is NOT a face.
Now match with vertex $1^-$. Given a face $\sigma$ unmatched by $\mathcal{M}_{1^+}$, we examine the two possible cases that would result in $\sigma$  NOT being  matched by $1^-$.
\begin{enumerate}
\item $1^-\notin \sigma$ but $1^-\cup \sigma$ is NOT a face: However, we then have $1^+, 1^-$ both not in $\sigma$, so $\sigma\subsetneq \{2^\pm,\ldots , n^\pm\}$. (Note that the latter set has full dimension, but is not a facet.) Hence for some $i,$ one of $i^+, i^-$ is not in $\sigma$. But then either $1^+\cup\sigma\in F_{(i^+, 1^-)^c}$ or $1^-\cup\sigma\in F_{(1^+, i^-)^c}$, contradicting the fact that $1^+\cup\sigma, 1^-\cup\sigma$ are not faces. 
This case is thus eliminated. 
\item $1^-\in \sigma$ but $\sigma$ cannot be matched with $\tau=\sigma\setminus 1^-$: This happens only if $\tau$ was already matched by $\mathcal{M}_{1^+}$, that is, $1^+\cup\tau$ is a face. 
\end{enumerate}

Hence the unmatched faces $\sigma$ after the element matchings $\mathcal{M}_{1^+},\mathcal{M}_{1^-}$, are precisely those described in Part (2) above. 

Since $1^+\cup (\sigma\setminus 1^-)$ is a face, there is at least one $i^+$, $i\ge 2$, such that 
$i^+\notin  1^+\cup (\sigma\setminus 1^-)$.  So   $i^+\notin 1^+\cup\sigma$ and thus 
$V^+\not\subseteq 1^+\cup \sigma$.
Similarly, since $\sigma$ is a face, there is at least one $l^-$, $l\ge 2$ such that $\ell^-\notin \sigma$. So $V^-\not\subseteq 1^+\cup\sigma$.

Since $1^+\cup\sigma$ is not a face, Equation~\eqref{eqn:nonface-prismDelta2} now implies that for some $j\ge 2$, $1^+\cup\sigma =V^+\cup V^- \setminus\{j^+,j^-\}$.
Conversely, it is clear that such faces $\sigma$ satisfy the conditions of Part (2).

Hence  there are exactly $(n-1)$ unmatched faces $\sigma$ all of the same size $(2n-3)$,  namely 
\[\sigma=\{1^+, 1^-, \ldots, n^+, n^-\}\setminus\{1^+, j^+, j^-\}, 2\le j\le n.\]
These are the critical cells of our Morse matching, and so, combining with a 0-cell because the empty set is matched with $1^+$, the statement about the homotopy type follows. 
\end{proof}

\begin{ex}\label{ex:simplicial-vertexS} 
We apply Theorem~\ref{thm: tot-cut complex add simplicial vertex -Lei-M-S} to compute the homotopy type of the graph $G'_n$ obtained by deleting one vertex, say $n^+$,  from the prism graph $G_n.$   %
Then $n^-$ is a simplicial vertex of $G'_n$, and thus, from the preceding theorem, we obtain the homotopy equivalence 

\[\Delta^t_2(G'_n)\simeq \susp \Delta^t_2( G_n\setminus\{n^+, n^-\}) =\susp \Delta^t_2( G_{n-1})\simeq\susp \bigvee_{n-2} \bbS^{2n-6}\simeq \bigvee_{n-2} \bbS^{2n-5}.\]
\end{ex}

We turn now to another important class of graphs.

\begin{df}[{\cite{WestGraphTheory1996}}] \label{defn:chordal}
    A graph is \emph{chordal} if it contains no induced cycle with more than 3 vertices.
    \end{df}
\begin{theorem}[Fr\"oberg {\cite{Froberg1990}, Eagon-Reiner \cite[Proposition~8]{EagonReiner1998}}]\label{thm:Froberg}
$\Delta_2^t(G) = \Delta_2(G)$ is vertex decomposable if and only if $G$ is chordal.
\end{theorem}
Recall that vertex decomposability implies shellability.  We show that for all $k$ the total $k$-cut complex of any chordal graph is vertex decomposable.  Chordal graphs are also studied in \cite{BDJRSX} where it is shown that 
the 3-cut complex of any chordal graph is shellable.  

One of the well-known properties of chordal graphs (see \cite{Golumbic1980}) 
is that there exists an \emph{elimination ordering}: that is, an ordering $v_1, v_2, ..., v_n$ of the vertices such that for all $i$ there are edges between all pairs of neighbors  of $v_i$ in $\{v_{i+1}, v_{i+2}, \dots , v_n\}$. In other words, vertex $v_i$ is a simplicial vertex of $G\setminus \{v_1,\dots , v_{i-1} \}$. For analogous work with the Alexander dual of $\Delta_k^t(G),$ see \cite[Theorem 1.5]{kim-lew}.

\begin{theorem}\label{thm:Lei-chordal-totk-vertdecomposable} If $G= (V,E)$ is a chordal graph, then $\Delta_k^t(G)$ is vertex decomposable, therefore shellable, for all $k\ge 2.$
\end{theorem}

\begin{proof}
We prove this by induction on $k$ and $n=|V|$. If $G$ has a single vertex, then $\Delta_k^t(G)$ is the void complex and thus vertex decomposable for all $k$. If $k=2$, then $\Delta_2^t(G)=\Delta_2(G)$, which is vertex decomposable for any $n$ as proved in \cite[Proposition~8]{EagonReiner1998}. Now let $k>2$ and $n>1$. Assume that $\Delta_i^t(G)$ is vertex decomposable for all chordal 
graphs on $m$ vertices for all $i<k$, and for all $i =k$ and $m<n$.

Let $G$ be chordal with $n$ vertices. The vertex decomposition of $\Delta_k^t (G)$ will follow the elimination ordering, $v_1, \dots, v_n$. The vertex $v_1$ is a simplicial vertex in $G$, therefore the set of its neighbors, $N(v_1)$, forms a clique. We need to show that both subcomplexes $\lk_{\Delta^t_k(G)} v_1$ and $\del_{\Delta^t_k(G)} v_1$ are vertex decomposable. By the inductive hypothesis and Lemma \ref{lem:link of faces, total cut}, we know that $\lk_{\Delta^t_k(G)} v_1 = \Delta_{k}^t(G\setminus v_1)$ is vertex decomposable. 

Now it suffices to show that $\del_{\Delta^t_k(G)} v_1$ is vertex decomposable. By Lemma \ref{lem:pure subcomplex avoiding one vertex, total cut}, $\del_{\Delta^t_k(G)} v_1$ is the star of the face $N(v_1)$ in the complex $\Delta^t_{k-1}(G\setminus v_1)$. By the inductive hypothesis, $\Delta^t_{k-1}(G\setminus v_1)$ is vertex decomposable. Since vertex decomposability is preserved under taking links and joins (see \cite[Prop.~2.3,\;2.4]{Provan-Billera}), it is also preserved under taking stars. Together with the inductive hypothesis, this shows $\del_{\Delta^t_k(G)} v_1$ is vertex decomposable.
\end{proof}

\begin{theorem}\label{homotopytype-chordal-totk-Mark}
Let $G$ be a chordal graph with $n$ vertices and $c$ connected components, and $k\geq 1$. Then $\Delta_k^t(G)$ is homotopy equivalent to the wedge of $\binom{c-1}{k-1}$ spheres of dimension $n-k-1$ if $c\geq k$. If $c<k$ it is contractible if $G$ has an independent set of size $k$; otherwise it is the void complex.
\end{theorem}
\begin{proof}
Let $G$ be a chordal graph with $n$ vertices and $c$ connected components, and $k\geq 1$. We know $\Delta_k^t(G)$ is shellable from Theorem \ref{thm:Lei-chordal-totk-vertdecomposable}, so it suffices to find its homology by Theorem \ref{thm:shell-implies-homotopytype}. As $G$ is chordal there is an elimination ordering of vertices;  furthermore, in any elimination ordering for $G$, for each component,
some vertex occurs latest in the order.  These $c$ vertices can be moved to the
end of the elimination ordering.  So we can assume that after $n-c$ eliminations, the resulting graph consists of $c$ isolated vertices.

Let $G_i$ be $G$ with the first $i$ vertices eliminated, that is, $G_i=G\setminus\{v_1,\ldots, v_i\}$.

If $c\geq k$, then $\Delta_k^t(G_{n-c})$ is nonempty and $\Tilde{H}_{c-k-1}(\Delta_k^t(G_{n-c}))=\mathbb{Z}^{\binom{c-1}{k-1}}$, and so by Theorem \ref{thm: tot-cut complex add simplicial vertex -Lei-M-S} $\Tilde{H}_{c-k-1+i}(\Delta_k^t(G_{n-c-i}))=\Tilde{H}_{c-k-i}(\Delta_k^t(G_{n-c-i+1}))=\mathbb{Z}^{\binom{c-1}{k-1}}$. In particular,
\[\Tilde{H}_{n-k-1}(\Delta_k^t(G))=\Tilde{H}_{n-k-1}(\Delta_k^t(G_{0}))=\mathbb{Z}^{\binom{c-1}{k-1}}.\]

If $c<k$, then $\Delta_k^t(G_{n-c})$ is the void complex. So by Theorem \ref{thm: tot-cut complex add simplicial vertex -Lei-M-S} $\Delta_k^t(G)$ is either contractible or the void complex, and it is only the void complex if $G$ has no independent sets of size $k$.
\end{proof}

For example, if $G$ is the path $P_n$ on $n$ vertices, $\Delta_k^t(P_n)$ is contractible if $2k\le n+1$ and is the void complex if $2k\ge n+2$.
More generally:
\begin{cor}\label{cor:trees} If $G$ is a tree on $n$ vertices with an independent set of size $k$, then $\Delta_k^t(G)$ is contractible.
\end{cor}

The situation for cycle graphs is quite a bit more complicated.
Consider the total cut complex $\Delta_n^t(C_{2n})$. It is easy to see that there are only two facets, $\{2i \mid 1\le i\le n\}$ and $\{2i-1 \mid 1\le i\le n\}$, since the complement of any other set of size $n$ will contain at least two consecutive vertices. These facets are disjoint, and hence the $(n-1)$-dimensional total cut complex $\Delta^t_n(C_{2n})$ is homotopy equivalent to one 0-sphere.  Note that in \cite{BDJRSX} it is shown that  the ordinary cut  complex 
$\Delta_n(C_{2n})$ is shellable and homotopy equivalent to a wedge of $(n-1)$-dimensional spheres. 
  By counting the number of $k$-subsets of $[n]$ containing at most one of $i, i+1$ (modulo $n$) for each $i$, we see that the number of facets of $\Delta_k^t(C_n)$ is $\tfrac{n}{n-k}\binom{n-k}{k}$.

\begin{theorem}\label{thm:total-cut-cycle} For $n<2k$, $\Delta_k^t(C_n)$ is the void complex and hence shellable.
For $n\ge 2k\ge 4$,
$\Delta_k^t(C_n)$ is homotopy equivalent to a single sphere in dimension $n-2k$, and hence it is not shellable.
\end{theorem}
\begin{proof}
Label the vertices of $C_n$ with $[n]=\{1,2,\ldots, n\}$ in standard cyclic order. An arbitrary set $\sigma \subseteq [n]$ is a face of $\Delta_k^t(C_n)$ if and only if $\alpha([n] \setminus \sigma) \ge k.$  In particular, if $n<2k$, then $\Delta_k^t$ is the void complex.  So assume $n\ge 2k$.

First, we introduce a few terms and observe some facts about the independent sets in $C_n$. For any $a \in [n]$ and $l \in \{1, \ldots, n-1\}$, the set of consecutive vertices $\{a, a+1, \ldots, a+l-1\}$ ($n+i$ is identified with $i$ for all $i \ge 1$) is called a {\em block of order $l$}. 
The independence number of a block is determined by its order. Namely, for any  block which is not equal to the entire set $[n]$,
\begin{equation}\label{alpha_for_block}
\alpha(\text{block of order } 2t-1) = t = \alpha(\text{block of order } 2t).
\end{equation}
Further, we say that two disjoint blocks are {\em independent} if there is no vertex from one block that is adjacent to a vertex from another block. In other words, two blocks are independent if they cannot be merged into one bigger block. Every nonempty $B \subsetneq [n]$ can be seen as  a disjoint union $B= B_1 \sqcup \ldots \sqcup B_m$, where $B_1, \ldots, B_m$ are pairwise independent blocks ($m\ge1$). This representation is unique (up to the order of blocks). From this representation we directly determine the independence number of $B$:  $\alpha(B) = \alpha(B_1) + \cdots + \alpha(B_m).$ We will also need the following  slightly more general additive relation. If we have a set $B \subsetneq [n]$, and a block $B'= \{a, a+1, \ldots, a+2t-1\}$ of order $2t$, such that $B \cap B' = \emptyset$ and  $\{a-1, a+2t\} \not\subseteq B$ (that is, $B'$ can be adjacent to some block from $B$, but it is not adjacent at both endpoints), then $\alpha(B \cup B')= \alpha(B) + \alpha(B') = \alpha(B) + t.$ We will refer to this fact as ``the addition of an even block.''

Thus, among any $2k -1$ vertices of $C_n$ ($n \ge 2k$) there are at least $k$ that are independent. Therefore, all $\sigma \subseteq [n]$ of cardinality  $|\sigma| \le n-(2k-1)$ belong to the complex $\Delta_k^t(C_n).$

We prove the theorem by using discrete Morse theory. First, construct the element matching $\mathcal{M}_1$  on the poset of $\Delta_k^t(C_n)$ using vertex 1. Let $K_1$ denote the set of unmatched faces after $\mathcal{M}_1$. An arbitrary face $\sigma \in \Delta_k^t(C_n)$ is in $K_1$ if and only if $ 1 \notin \sigma$  and $\sigma \cup \{1\} \notin  \Delta_3^t(C_n).$ This can be equivalently stated in terms of independence numbers:
$\alpha([n] \setminus \sigma) \ge k$, but $\alpha([n] \setminus (\sigma \cup \{1\} )) \le k-1.$ Removing one vertex from a set can decrease the independence number at most by 1, so we know precisely that $\sigma \in K_1$ if and only if $1 \not \in \sigma$ and the following two conditions hold:
\begin{gather}
\alpha\big([n] \setminus \sigma\big) = k \label{first_ind_cond} \\
\alpha\big([n] \setminus (\sigma \cup \{1\} )\big) = k-1 \label{second_ind_cond}
\end{gather}

For each $\sigma \in K_1$, consider the unique block representation: $[n] \setminus (\sigma \cup \{1\}) = B_1 \sqcup \cdots \sqcup B_m$ ($m \ge 1$), with blocks ordered by their smallest elements. The block $B_1$ which contains the smallest number in $[n] \setminus (\sigma \cup \{1\})$ will be called {\em the first block in}  $[n] \setminus (\sigma \cup \{1\})$, or {\em the first block missing from} $\sigma \cup \{1\}$. Also, $\sigma \in K_1$ will be called {\em regular} if the first block missing from $(\sigma \cup \{1\})$ does not contain $n$ (informally, the first block is not ``at the end''). An important observation is that there is only one face in $K_1$ which is not regular. Indeed, if $\sigma \in K_1$ is not regular, $n \in B_1$ implies that $m=1,$ and that 
$$[n] \setminus (\sigma \cup \{1\}) = B_1 = \{ n-l, n-(l-1), \ldots, n\},$$
for some $l$. From conditions (\ref{first_ind_cond}) and (\ref{second_ind_cond}), and relation (\ref{alpha_for_block}), we conclude that $l=2k -3.$ Therefore, the only not regular face is
\begin{equation*}%
\widetilde{\sigma} = [n] \setminus \{n-(2k-3), \ldots, n, 1\} = \{2, 3,\ldots, n-(2k-2)\}. 
\end{equation*}
All regular faces can be  partitioned into two categories, based on the parity of the first missing block:
$$\mathcal{O} = \{\sigma \in K_1 \mid \text{$\sigma$ is regular and the first block in $[n] \setminus (\sigma \cup \{1\})$ has odd order}\},$$
$$\mathcal{E} = \{\sigma \in K_1 \mid \text{$\sigma$ is regular and the first block in $[n] \setminus (\sigma \cup \{1\})$ has even order}\}.$$

We will proceed with element matchings (using vertices $2, 3, \ldots$). Our idea is to prove that in each pair we will have a simplex $\sigma$ from $\mathcal{E}$, and the simplex $\sigma \cup \{a\}$ from $\mathcal{O}$, where $a$ is the smallest number not in $\sigma \cup \{1\}$. First, we prove that there is a bijection from $\mathcal{E}$ to $\mathcal{O}$. Let $\sigma \in \mathcal{E}$ be an arbitrary simplex, and let 
$$[n]\setminus (\sigma \cup \{1\}) = {\underbrace{\{a, a+1, \ldots, a+2t-1\}}_{B_1}} \sqcup  B_2 \sqcup \cdots \sqcup B_m$$
be the block representation with the first block $B_1$  ($t \ge 1, a \ge 2$). We define:
$$f: \mathcal{E} \to \mathcal{O},\ \ f(\sigma) \coloneqq \sigma \cup \{a\}.$$
Then $[n]\setminus (f(\sigma) \cup \{1\})$  has the following block representation: 
$$[n]\setminus (f(\sigma) \cup \{1\}) =\{ a+1,\ldots, a+2t-1\} \sqcup B_2 \sqcup \cdots \sqcup B_m.$$
We need to prove that $f$ is a well-defined function, and then that it is a bijection.
\medskip

\begin{description}
    \item[$f$ is well-defined] for an arbitrary $\sigma \in \mathcal{E}$ we show that $f(\sigma) \in \mathcal{O}.$ 
    
\noindent It is obvious that $f(\sigma)$ is regular and that the first missing block has odd order. It remains to prove that $f(\sigma)$ indeed belongs to $K_1$, i.e.,\ that $f(\sigma)$ satisfies conditions (\ref{first_ind_cond}) and (\ref{second_ind_cond}). First, since face $\sigma$ satisfies 
(\ref{first_ind_cond}), we know that: 
\begin{center}
    $  \alpha \big([n] \setminus \sigma\big) = \alpha\big(\{1\} \cup \{a, a+1, \ldots, a+2t-1\} \sqcup B_2 \sqcup \cdots \sqcup B_m \big) = k.$
\end{center}
For $f(\sigma)$ the corresponding independence number is:
\begin{center}
  $\alpha \big([n] \setminus f(\sigma)\big) = \alpha\big(\{1\} \cup \{a+1, \ldots, a+2t-1\} \sqcup B_2 \sqcup \cdots \sqcup B_m \big).$
\end{center}
For the sake of clarity, if $n \in B_m$ then $\{1\}$ is merged with $B_m$ both in $[n] \setminus \sigma$ and $[n] \setminus f(\sigma)$, but it does not influence the following reasoning about independent blocks. 
Since $\{a, a+1, \ldots, a+2t-1\}$ and $\{a+1, \ldots, a+2t-1\}$ have equal independence numbers, and the independence number is additive for independent blocks, for $a \ge 3$ it is obvious that 
\begin{equation}\label{first_for_s_fs}
    \alpha \big([n] \setminus f(\sigma)\big) = \alpha \big([n] \setminus \sigma\big) =k.
\end{equation}
The case $a=2$ is slightly different because $\{1\}$  is merged with $\{2, 3, \ldots, 2+2t-1\}$ in $[n] \setminus \sigma.$ However, we observed at the beginning  that the addition of an even block (not adjacent to the current set at both ends) properly increases the independence number, so we have:
\begin{equation*}
   \alpha \big([n] \setminus \sigma\big) = t + \alpha\big(\{1\} \sqcup B_2 \sqcup \cdots \sqcup B_m \big) = k,
\end{equation*}
which directly implies 
\begin{equation*}
   \alpha \big([n] \setminus f(\sigma)\big) = t + \alpha\big(\{1\} \sqcup B_2 \sqcup \cdots \sqcup B_m \big) = k.
\end{equation*}
This proves relation (\ref{first_for_s_fs}) for all $\sigma \in \mathcal{E}$, so $f(\sigma)$ satisfies (\ref{first_ind_cond}). The second relation (\ref{second_ind_cond}) for $f(\sigma)$ follows directly from (\ref{second_ind_cond}) for $\sigma$:
\begin{align*}
k-1 &= \alpha\big([n] \setminus (\sigma\cup \{1\})\big)\\
 &= \alpha\big(\{a, a+1, \ldots, a+2t-1\} \sqcup B_2 \sqcup \cdots \sqcup B_m \big) \\
 &= \alpha\big(\{a+1, \ldots, a+2t-1\} \sqcup B_2 \sqcup \cdots \sqcup B_m \big) \\
 &=\alpha\big([n] \setminus (f(\sigma)\cup \{1\})\big).
\end{align*}
This finishes the proof that $f$ is well-defined.
\medskip
\item[$f$ is bijective] it is obvious that $f$ is one-to-one, we only need to prove that $f$ is onto.  Consider an arbitrary  $\tau \in \mathcal{O}$ and the block representation for $[n] \setminus (\tau \cup \{1\})$:
$$[n]\setminus (\tau \cup \{1\}) = {\underbrace{\{b, b+1, \ldots, b+2t\}}_{B_1}} \sqcup B_2 \sqcup \cdots \sqcup B_m,$$
with the first block $B_1$, $b \ge 2$ and $t \ge 0.$ We first prove that $b \ge 3$. Suppose the contrary, that 
$[n]\setminus (\tau \cup \{1\}) = {\underbrace{\{2, 3, \ldots, 2+2t\}}_{B_1}} \sqcup B_2 \sqcup \cdots \sqcup B_m$ 
($2 \notin \tau$). Then $\alpha ([n]\setminus (\tau \cup \{1\})) = t+1 + \alpha(B_2) + \cdots + \alpha(B_m).$ Also, we have: $$[n]\setminus \tau =\{1, 2, \ldots, 2+2t\} \sqcup B_2 \sqcup \cdots \sqcup B_m.$$
This does not need to be a disjoint block representation (because the first block might be merged with $B_m$),  but since the independence number behaves properly in regard to the addition of an even block, we know: 
$$\alpha ([n]\setminus \tau) = t+1 + \alpha(B_2) + \cdots + \alpha(B_m) = \alpha ([n]\setminus (\tau \cup \{1\})),$$
which contradicts relations (\ref{first_ind_cond}) and (\ref{second_ind_cond}). Therefore, we proved that $b$ must be at least $3$. This allows us to define set $\sigma_{\tau}\coloneqq \tau \setminus \{b-1\}.$ Consider the block representation: 
$$[n]\setminus (\sigma_{\tau} \cup \{1\}) =\{b-1, b, b+1, \ldots, b+2t\} \sqcup B_2 \sqcup \cdots \sqcup B_m.$$
From the fact that $\tau$ belongs to $K_1$ it follows that $\sigma_{\tau}$ belongs to $K_1$ as well. (We omit the details because the reasoning is the same as in the proof that $f$ is well-defined.) Also, $\sigma_{\tau}$ is regular (because $\tau$ is regular) and the first block in $[n]\setminus (\sigma_{\tau} \cup \{1\})$ has even order, so $\sigma_{\tau} \in \mathcal {E}.$ Finally, we directly observe that $f(\sigma_{\tau}) = \tau,$ and the claim that $f$ is surjective is proved. 
\end{description}
We are ready to apply the element matchings $\mathcal{M}_2, \mathcal{M}_3, \ldots, \mathcal{M}_{n-2k +2}$ using the  vertices $2, 3, \ldots, {n-2k +2}$, respectively. We claim that for each $a \in \{2, 3, \ldots, n-2k +2\},$ matching $\mathcal{M}_a$ makes exactly the following pairs:
$$\{ (\sigma, f(\sigma)) = (\sigma, \sigma \cup \{a\}) \mid \sigma \in \mathcal{E} \text{ and } a \text{ is the minimal element in } [n]\setminus(\sigma \cup \{1\})\}.  $$
 Let $\sigma \in \mathcal{E}$ be an arbitrary face and let $a$ be  the minimal element in $[n]\setminus(\sigma \cup \{1\})$. A direct computation shows that $a \in \{2, 3, \ldots, n-2k+2\}.$ For $a\ge 3$, let $i \in \{2, \ldots, a-1\}$ be an arbitrary index. We want to prove that both $\sigma$ and $f(\sigma)$ are unpaired after $\mathcal{M}_i.$
From $i < a$ we know that $i \in \sigma$ and $i \in f(\sigma),$ so $\sigma$ cannot be matched with $\sigma \cup \{i\}$ and $f(\sigma)$ cannot be matched with $f(\sigma) \cup \{i\}.$ Further, $\sigma$ cannot be matched with $\sigma \setminus \{i\}$ because we will show that $\sigma \setminus \{i\}$ does not belong to $K_1$. Namely, consider the block representation for $[n] \setminus (\sigma \cup \{1\})$:
$$[n]\setminus (\sigma \cup \{1\}) =\{a, a+1, \ldots, a+2t-1\} \sqcup B_2 \sqcup \cdots \sqcup B_m.$$
From (\ref{second_ind_cond}) we get: $k-1 = \alpha([n]\setminus (\sigma \cup \{1\})) = t + \alpha(B_2) + \cdots + \alpha(B_m).$
Also, we have $[n] \setminus ((\sigma \setminus \{i\}) \cup \{1\}) = \{i\} \cup \{a, a+1, \ldots, a+2t-1\} \sqcup B_2 \sqcup \cdots \sqcup B_m.$ The addition of the even block $\{a, a+1, \ldots, a+2t-1\}$ increases the independence number properly, so we obtain:
$$
\alpha\big( [n] \setminus ((\sigma \setminus \{i\}) \cup \{1\})\big) = \alpha(\{i\}) + t + \alpha(B_2) + \cdots + \alpha(B_m) = 1+ k-1 = k,
$$
which contradicts relation (\ref{second_ind_cond}) for $\sigma \setminus \{i\},$ so $\sigma \setminus \{i\} \notin K_1.$ An analogous reasoning shows that set  $f(\sigma)\setminus \{i\}$ does not belong to $K_1$ either. 
This finishes the proof that  $\sigma$ and $f(\sigma)$ are unpaired after all element matchings $\mathcal{M}_2,  \ldots, \mathcal{M}_{a-1},$ so it is obvious that  these two faces will be matched together in $\mathcal{M}_{a}.$ 

We conclude that after all matchings $\mathcal{M}_2, \ldots, \mathcal{M}_{n-2k+2},$ all regular faces are paired off in pairs $(\sigma, f(\sigma)).$ There is only one critical cell --- the  irregular face $\widetilde{\sigma}$. By Theorem~\ref{thm:seq-of-elt-match}, the union of element matchings 
$\mathcal{M}_1 \cup \cdots \cup \mathcal{M}_{n-2k+2}$ is an acyclic matching on the face poset of $\Delta_k^t(C_n)$.  Therefore, the complex $\Delta_k^t(C_n)$  is homotopy equivalent to a CW complex with  one cell of dimension $n-2k$ (since $|\widetilde{\sigma}| = n-2k +1 $) and one $0$-cell (matched to the empty set). It follows that  $\Delta_k^t(C_n) \simeq S^{n-2k}.$ Also, since $\dim \Delta_k^t(C_n) = n-k-1$, $\Delta_k^t(C_n)$ is not a shellable complex.
\end{proof}
The Morse matching of the proof can be modified to give a Morse matching for the total cut complex of a path with $n$ vertices, which we observed is contractible when $2k\le n+1$.

\subsection{$\Delta_2^t$ for certain classes of graphs}\label{subsec:Delta2}

For a few classes of graphs, we have results for the total cut complex only for $k=2$.  

First we consider the \emph{grid graph} $G(m,n)$, defined as the Cartesian product of the paths $P_m$ and $P_n$ with $m$ and $n$ vertices respectively. Picture the $m\times n$ grid graph as the set of lattice points with coordinates $(i,j)$ ($1\le i\le m$, $1\le j\le n$), with horizontal and vertical line segments connecting them.  We find it convenient to define the graph with vertices labeled 1 to $mn$ as follows.  The vertex set of $G$ is $[mn]$, with vertices grouped in rows as $\{(i-1)n+j: 1\le j\le n\}$ for fixed $i$, $1\le i\le m$.  The edges are of two types:
``horizontal'' edges are pairs of vertices of the form $\{(i-1)n+j, (i-1)n+j+1\}$, where $1\le i\le m$ and $1\le j\le n-1$; ``vertical'' edges are pairs of the form $\{(i-1)n+j,in+j\}$, where $1\le i\le m-1$ and $1\le j\le n$. See Figure~\ref{fig:GridGraph33}.  

From Definition~\ref{def:tot-cut-cplx}, 
when $k=1$ the total cut complex $\Delta^t_1 (G(m,n))$ is simply the boundary of an $(mn-1)$-simplex, so it is shellable and has the homotopy type of a single sphere in dimension $mn-2$. 

\begin{theorem}\label{thm:grid-k=2MyMorseMatching2021Sept} For $n, m\ge 2,$ the $(mn-3)$-dimensional total cut complex $\Delta_2^t (G(m,n))$ has  homotopy type
\[  \bigvee_{(m-1)(n-1)} \mathbb{S}^{mn-4}.\]
\end{theorem}

\begin{proof}  Let $n\ge m\ge 2$ and write $G=G(m,n).$   Label the vertices of $G$  as above.
Thus 2 and $n+1$ are the two neighbors of 1.

 We note two easy facts:
\begin{enumerate}
\item The number of edges in $G$ is $m(n-1)+n(m-1)=2mn-m-n,$ by counting horizontal and vertical edges in the grid graph.
\item The facets of $\Delta_2^t(G)$ are the sets $[mn]\setminus \{a,b\},$ where $\{a, b\}$ is NOT an edge.  Equivalently, there are exactly $2mn-m-n$  non-facets, namely  the sets $[mn]\setminus\{e_1, e_2\},$ where $\{e_1, e_2\}$ is an edge.
\end{enumerate}

We will construct a sequence of element matchings on the face poset of $\Delta_2^t(G),$ beginning with the vertex 1.
 First note that the fact that $G$ is triangle-free implies that $\Delta_2^t(G)$ contains all subsets of size less than or equal to $mn-3$.  This means that for the element matching $M_1$ with vertex 1, which matches a pair $(\sigma, \sigma\cup\{1\})$ where $1\notin \sigma$,  the faces that remain unmatched  must have size $(mn-2)$ or $(mn-3).$   The unmatched faces are those $\sigma$  where $1\notin \sigma$ and $\sigma\cup\{1\}\notin \Delta_2^t(G),$ so $\sigma\cup\{1\}=[mn]\setminus\{e_1, e_2\}$ where  $\{e_1, e_2\}$ is an edge.

Hence the unmatched faces are of the following two types:
\begin{description}
\item[Type 1] The size $(mn-3)$ subsets $A_{e_1,e_2}:=[mn]\setminus \{1\}\setminus \{e_1,e_2\}$ for each edge $\{e_1<e_2\}$ not containing 1, so that we may assume $1< e_1<e_2;$
there are precisely $2mn-m-n-2$ of these; and
\item[Type 2] The size $(mn-2)$ facets $B_a:=[mn]\setminus \{1,a\}$ for each NON-edge $ \{1,a\},$ i.e., $a\notin\{1,2,n+1\};$  there are $mn-3$ of these.
\end{description}

Note that each $a\ge 2$ yields two unmatched faces $A_{a, e_2},$ corresponding to $e_2=a+1$ provided $a\not\equiv 0 \mod n$ (i.e., $a$ is not in the right-most column)
and $e_2=a+n$ provided $a\le (m-1)n,$ (i.e., $a$ is not in the bottom row).  %

For each $a=2,\dots, (m-1)n,$ consider the sequence of element matchings $M_a.$ We have
\[\begin{cases} A_{a,a+1}\rightarrow A_{a,a+1}\cup\{a\}=B_{a+1}, & a+1\le n,\\
                         A_{a,a+n}\rightarrow A_{a,a+n}\cup\{a\}=B_{a+n}, & a\notin\{(m-1)n+1, \dots, mn\}.
\end{cases}\]
Thus if $a+1=qn+r, 0\le r\le n-1, q\ge 1,$ then $B_{a+1}=[mn]\setminus\{1,a+1\}$ is matched with $A_{a+1-n, a+1}.$
More generally, for  each $a\notin \{1,2,n+1\},$ we see that:
\begin{enumerate}
\item  If $a$ has two neighbors preceding it, $e_1<e_2<a,$ $B_a$ is matched with $A_{e_1,a}.$
\item If $a$ has only one neighbor preceding it, $e_1<a,$ this means $a$ is in column 1, $a\ne n+1,$ so $e_1=a-n$ and $B_a$ is matched by $A_{a-n,a}.$   In particular, $B_{mn}$ is matched with $A_{(m-1)n,n}.$
\item Thus, except for the first row, the matching uses up Type 1 faces corresponding to all the vertical edges in the grid graph.
\end{enumerate}

This establishes that all the Type 2 unmatched faces are matched by the sequence of element matchings $\{M_a\}, a=1,2,\dots, (m-1)n.$
The union of the matchings $\cup_{a=1}^{(m-1)n} M_a$ is an acyclic matching by Theorem~\ref{thm:seq-of-elt-match}.
 The remaining unmatched faces are all of Type 1, giving critical cells of size $(mn-4)$ along with one 0-cell (matched to the empty set). From the description of the Type 1 and Type 2 facets, we see that the number of critical cells of size $(mn-4)$ is
\[2mn-m-n-2- (mn-3)=mn-m-n+1=(m-1)(n-1).\]
Another way to obtain this count is to observe that the critical cells correspond to the $(n-1)$ horizontal edges in each of the $m$ rows of  the grid graph, except the top row: this gives precisely $(m-1)(n-1).$

It follows that the homotopy type of the cut complex $\Delta_2(G)$ is a wedge of $(m-1)(n-1)$ spheres of dimension $(mn-4),$ one lower than the top dimension.
\end{proof}

An example of the Morse matching used in the above theorem follows:

\begin{figure}[htb]
\begin{center}
\begin{tikzpicture}
\node (1) at (1,3) {$\circ$};
\node at (.8,3) {$\textcolor{blue}{1}$};
\node (2) at (2,3) {$\circ$};
\node  at (1.8,3) {$\textcolor{blue}{2}$};
\node (3) at (3,3) {$\circ$};
\node at  (2.8,3) {$\textcolor{blue}{3}$};
\node (4) at (1,2) {$\circ$};
\node at  (.8,2) {$\textcolor{blue}{4}$};
\node (5) at (2,2) {$\circ$};
\node at  (1.8,2) {$\textcolor{blue}{5}$};
\node (6) at (3,2) {$\circ$};
\node at  (2.8,2) {$\textcolor{blue}{6}$};
\node (7) at (1,1) {$\circ$};
\node at  (.8,1) {$\textcolor{blue}{7}$};
\node (8) at (2,1) {$\circ$};
\node at  (1.8,1) {$\textcolor{blue}{8}$};
\node (9) at (3,1) {$\circ$};
\node at  (2.8,1) {$\textcolor{blue}{9}$};
 \draw  (1) -- (2) -- (3); \draw  (4) -- (5) -- (6);   \draw (7) -- (8) -- (9); 
\draw   (1) -- (4) -- (7); \draw  (2) -- (5) -- (8);   \draw (3) -- (6) -- (9); 
\end{tikzpicture}
\end{center}
\caption{The Grid Graph  $G(3,3)$} 
\label{fig:GridGraph33}
\end{figure}
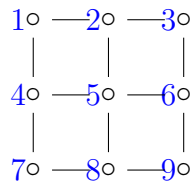

\begin{ex}\label{ex:forDMTgridgraph} The unmatched codimension-one faces of  the cut complex $\Delta_2(G)$ for the  grid graph $G=G(3,3)$, see Figure~\ref{fig:GridGraph33},    
 are indicated by an asterix:
\begin{align*}  A_{23}&=V\setminus\{1\}\setminus\{2,3\} \rightarrow B_3=V\setminus\{1,3\}\\
                       A_{25}&=V\setminus\{1\}\setminus\{2,5\} \rightarrow B_5=V\setminus\{1,5\}\\
                       A_{36}&=V\setminus\{1\}\setminus\{3,6\} \rightarrow B_6=V\setminus\{1,6\}\\
                     A_{45}&=V\setminus\{1\}\setminus\{4,5\}   (*)\\
                      A_{47}&=V\setminus\{1\}\setminus\{4,7\} \rightarrow B_7=V\setminus\{1,7\}\\
                       A_{56}&=V\setminus\{1\}\setminus\{5,6\} (*)\\
                       A_{58}&=V\setminus\{1\}\setminus\{5,8\} \rightarrow B_8=V\setminus\{1,8\}\\
                       A_{69}&=V\setminus\{1\}\setminus\{6,9\} \rightarrow B_9=V\setminus\{1,9\}\\
                       A_{78}&=V\setminus\{1\}\setminus\{7,8\} (*)\\
                       A_{8,9}&=V\setminus\{1\}\setminus\{8,9\} (*)
\end{align*}
\end{ex}

Next we consider  graphs obtained by adding edges to cycles.
\begin{df}  The  {\em squared cycle graph} $W_n$ is the graph with vertex set $[n],$ and edge-set
$\{ (i, i+1 \mod n), (i, i+2 \mod n)\},$  $i=1, \dots n.$  
\end{df}

We label the vertices of $W_n$ in order with $\{1,2,\dots, n\}.$
If $n\le 5,$ $W_n$ is the complete graph $K_n.$ 
\begin{prop}\label{prop:MarijaDMT-Delta2-wreath-graph}
The 3-dimensional total cut complex $\Delta_2^t(W_6)$ is homotopy equivalent to~$\mathbb{S}^1.$
For $n\ge 7$ the $(n-3)$-dimensional total cut complex $\Delta_2^t(W_n)$ has the homotopy type   of~$\mathbb{S}^{n-4},$ one sphere in dimension one lower than the top. 
 Hence for $n\ge 6,$ the total cut complex $\Delta_2^t(W_n)$ is not shellable.
\end{prop}
\begin{proof} If $n=6,$ $\Delta_2^t(W_6)$ has facets $\{ 1245, 2356, 3461\}$, and it is not hard to see that it is homotopy equivalent to $\mathbb{S}^1$ from its face lattice, for example. 

For $n\ge 7,$ we exhibit a Morse matching.

We denote the set of vertices of $W_n$ by $[n]$ and label the vertices by $1, 2, \dots, n$ in cyclic order. As in the case of cycles, each facet of the complex $\Delta_2^t(W_n)$ has $n-2$ vertices, so $\dim (\Delta_2^t(W_n))=n-3$. Also, among any four vertices of $W_n$ (for $n \ge 7$), there are at least two which are not connected, so every subset of $\{1, 2, \dots, n\}$ of cardinality at most   $n-4$ is a face of $\Delta_2^t(W_n)$.

Let $\mathcal{M}_1$ be the element matching on the face poset of $\Delta_2^t(W_n)$ using vertex 1. A face $\sigma \in \Delta_2^t(W_n)$ is unmatched if and only if $1 \notin \sigma$ and
 $\sigma \cup \{1\} \notin \Delta_2^t(W_n).$ Since the complete $(n-5)$-skeleton is  in  the complex $\Delta_2^t(W_n)$, the  unmatched faces can contain $n-4, n-3$ or $n-2$ vertices.

\begin{itemize}
\item 
The unmatched faces of cardinality $n-4$ are complements of the union of $\{1\}$ and a set of three adjacent vertices $k, k+1, k+2$. There are $n-3$ faces of this type:
 $$A_k= [n] \setminus \{1,k, k+1, k+2\}, \ \text{ for } \ k \in \{2,3, \dots, n-2\}.$$

\item 
There are two types of unmatched faces of cardinality $n-3$. The first type are faces that are complements of the union of $\{1\}$ and a set of two adjacent vertices $k, k+1$.  There are $n-4$ faces of this type, and we denote them by:
 $$B_k= [n] \setminus \{1, k+1, k+2\}, \ \text{ for } \ k \in \{2, 3,\dots, n-3\}.$$
 Observe that we do not have $[n] \setminus \{1, 2, 3\}, [n] \setminus \{1, n-1, n\}$ and $[n] \setminus \{n, 1, 2\}$ here because these sets are not faces of the complex.
\item 
There are $n-5$ faces of cardinality  $n-2$ which do not contain vertex 1, and all of them are unmatched:
$$D_k= [n] \setminus \{1, k+2\}, \ \text{ for } \ k \in \{2, 3,\dots, n-4\}.$$
\end{itemize}

We proceed with the following element matchings:
$$\mathcal{M}_2, \mathcal{M}_3, \dots, \mathcal{M}_{n-3},$$
where  $\mathcal{M}_k$ denotes the element matching using vertex $k \in \{2, 3, \dots, n-3\}.$ For all $k \le n-4$, the matching $\mathcal{M}_k$ results in exactly two pairings:
$$(A_k, B_k) \text{ and } (C_k, D_k),  \text{ because }  A_k \cup \{k\}= B_k \text{ and } C_k \cup \{k\}= D_k.$$
For $k=n-3$, the matching $\mathcal{M}_{n-3}$ makes only one pairing: $(A_{n-3}, B_{n-3})$. After these matchings, the unmatched faces are:
 $$A_{n-2} = [n] \setminus \{1, n-2, n-1, n\},$$
 $$C_{n-3} = [n] \setminus \{1, n-3, n-1\},$$
 $$C_{n-2} = [n] \setminus \{1, n-2, n\}.$$
 We apply one more element matching  $\mathcal{M}_{n-1}$ using element $n-1$, which gives the last pairing $(A_{n-2}, C_{n-2})$.

 By Theorem~\ref{thm:seq-of-elt-match}, the union of element matchings $\mathcal{M}_1 \cup \mathcal{M}_2 \cup \cdots \mathcal{M}_{n-3} \cup \mathcal{M}_{n-1}$ is an acyclic matching on the face poset of complex $\Delta_2^t(W_n)$. There is one critical cell $C_{n-3}$  and one $0$-cell (matched to the empty set).  Since the set $C_{n-3}$ has cardinality $n-3$, complex $\Delta_2^t(W_n)$ is homotopy equivalent to a sphere of dimension $n-4$.
\end{proof}

\section{Further directions}
In \cite{BDJRSX} we have further results on the cut complexes of  these last two  classes of graphs.  Here we present some data and conjectures on total cut complexes for these graphs.

In the last section we found the homotopy type of $\Delta_2^t(W_n)$, for the squared cycle graph.

 The following conjecture is supported by Sage data for $k\le 7$ and $n\le 21$:
\begin{conj}\label{conj:total-cut-wreath}
The $(n-k-1)$-dimensional complex $\Delta_k^t(W_n)$ is the void complex if $n\le 3k-1$. Otherwise it is homotopy equivalent to a single sphere in dimension 
\[\begin{cases} 2i+1, & n=3k+i, \ 0\le i\le k-1,\\
               2k+i, & n=4k+i, \ i\ge 0,
\end{cases} \]
and hence it is not shellable.
\end{conj}
For the grid graph $G(m,n)$, the homotopy type of  $\Delta^t_k(G(m,n)$ for $k=2$ was determined in 
 Subsection~\ref{subsec:Delta2}. The following data for the Betti numbers 
 for arbitrary $k$ suggests that the homotopy type is always that of spheres in a single dimension. These numbers were obtained using a sequence of Morse element matchings, in the standard (lexicographic) order of the vertices of $G(m,n)$.
\begin{table}[htbp]
\begin{center}
\scalebox{.8}{
\begin{tabular}{|c|c|c|c|c|c|c|c|}
\hline
$k\backslash n$ & 2 & 3 & 4 & 5 & 6 & 7 & 8 \\ \hline
$1$ & $\beta_2=1$ & $\beta_4=1$ & $\beta_6=1$ & $\beta_8=1$ & $\beta_{10}=1$ & $\beta_{12}=1$ & $\beta_{14}=1$ \\
$2$ & $\beta_0=1$ & $\beta_2=2$ & $\beta_4=3$ & $\beta_6=4$ & $\beta_8=5$ & $\beta_{10}=6$ & $\beta_{12}=7$ \\
$3$ & void & $\beta_0=1$ & $\beta_2=3$ & $\beta_4=6$ & $\beta_6=10$ & $\beta_8=15$ & $\beta_{10}=21$ \\
$4$ & void & void & $\beta_0=1$ & $\beta_2=4$ & $\beta_4=10$ & $\beta_6=20$ & $\beta_8=35$ \\
$5$ & void & void & void & $\beta_0=1$ & $\beta_2=5$ & $\beta_4=15$ & $\beta_6=35$ \\
$6$ & void & void & void & void & $\beta_0=1$ & $\beta_2=6$ & $\beta_4=21$ \\
$7$ & void & void & void & void & void & $\beta_0=1$ & $\beta_2=7$ \\
$8$ & void & void & void & void & void & void & $\beta_0=1$ \\
\hline
\end{tabular}
}
\end{center}
\caption{\small Nonzero Betti numbers for $\Delta^t_k (G(2,n)), 1\le k\le 8$.  The dimension of the total cut complex is $2n-k-1$.  Note that $\Delta^t_1(G(2,n))\simeq \bbS^{2n-2}$.}
\end{table}
\begin{table}[htbp]
\begin{center}
\scalebox{.8}{
\begin{tabular}{|c|c|c|c|c|c|c|l|}
\hline
$k \backslash n$ & 3 & 4 & 5 & 6 & 7 \\ \hline
$1$ & $\beta_7=1$ & $\beta_{10}=1$ & $\beta_{13}=1$ & $\beta_{16}=1$ & $\beta_{19}=1$ \\
$2$ & $\beta_5=4$ & $\beta_8=6$ & $\beta_{11}=8$ & $\beta_{14}=10$ & $\beta_{17}=12$ \\
$3$ & $\beta_3=6$ & $\beta_6=15$ & $\beta_9=28$ & $\beta_{12}=45$ & $\beta_{15}=66$ \\
$4$ & $\beta_1=4$ & $\beta_4=20$ & $\beta_7=56$ & $\beta_{10}=120$ & $\beta_{13}=220$ \\
$5$ & $\beta_i=0, i \geq 0$ & $\beta_2=13$ & $\beta_5=67$ & $\beta_8=206$ & $\beta_{11}=490$ \\
$6$ & void & $\beta_0=1$ & $\beta_3=42$ & $\beta_6=225$ & $\beta_9=748$ \\
$7$ & void & void & $\beta_1=7$ & $\beta_4=139$ & $\beta_7=761$ \\
$8$ & void & void & $\beta_i=0, i \geq 0$ & $\beta_2=33$ & $\beta_5=468$ \\
$9$ & void & void & void & $\beta_0=1$ & $\beta_3=135$ \\
$10$ & void & void & void & void & $\beta_1=10$ \\
$11$ & void & void & void & void & $\beta_i=0, i \geq 0$ \\
\hline
\end{tabular}
}
\end{center}
\vskip .1in
\caption{\small Nonzero Betti numbers for $\Delta^t_k (G(3,n))$, $1 \le k\le 11$.  The dimension of the total cut complex is $3n-k-1$.  Note that $\Delta^t_1(G(3,n))\simeq \bbS^{3n-2}$.
}
\end{table}
%
\begin{table}[htbp!]
\begin{center}
\scalebox{.8}{
\begin{tabular}{|c|c|c|c|c|c|c|l|}
\hline
$k\backslash n$ & 4 & 5 \\ \hline
$1$ & $\beta_{14}=1$ & $\beta_{18} = 1$\\
$2$ & $\beta_{12}=9$ & $\beta_{16} = 12$ \\
$3$ & $\beta_{10}=36$ & $\beta_{14} = 66$ \\
$4$ & $\beta_8=84$ & $\beta_{12} = 220$ \\
$5$ & $\beta_6=122$ & $\beta_{10} = 489$ \\
$6$ & $\beta_4=102$ & $\beta_8 = 737$ \\
$7$ & $\beta_2=29$ & $\beta_6 = 705$ \\
$8$ & $\beta_0=1$ & $\beta_4 = 340$ \\
$9$ & void & $\beta_2 = 53$ \\
$10$ & void & $\beta_0 = 1$ \\
\hline
\end{tabular}
}
\end{center}
\caption{\small Nonzero Betti numbers for $\Delta^t_k (G(4,n))$, $1\le k\le 10$.  The dimension of the total cut complex is $4n-k-1$.  Note that $\Delta^t_1(G(4,n))\simeq \bbS^{4n-2}$.
}
\end{table}
%

The data supports the following conjectures:
\begin{conj}\label{conj:Grid-total}  The nonzero Betti numbers satisfy the following formulas:
\[\beta_{2n-2k}(\Delta^t_k(G(2, n)))=\binom{n-1}{k-1},\ k\ge 2,\] 
\[\beta_{3j}(\Delta^t_3(G(3, 2+j)))=\binom{2j+2}{2},\ (k=3), \quad   
\beta_{1+3j}(\Delta^t_4(G(3, 3+j)))=\binom{2j+4}{3},\ (k=4).\]
Equivalently,
\[\beta_{3n-6}(\Delta^t_3(G(3, n)))=\binom{2n-2}{2},\ (k=3), \quad   
\beta_{3n-8}(\Delta^t_4(G(3, n)))=\binom{2n-2}{3},\ (k=4).\]
\end{conj}

\section{Acknowledgments}

We thank the organizers of the 2021 Graduate Research Workshop in Combinatorics, where this work originated. We also thank Natalie Behague, Dane Miyata, and George Nasr for their early contributions to our project. Marija Jeli\'c Milutinovi\'c has been supported by the Project No.\ 7744592 MEGIC ``Integrability and Extremal Problems in Mechanics, Geometry and Combinatorics'' of the Science Fund of Serbia, and by the Faculty of Mathematics University of Belgrade through the grant (No.\ 451-03-68/2022-14/200104) by the Ministry of Education, Science, and Technological Development of the Republic of Serbia. Rowan Rowlands was partially supported by a graduate fellowship from NSF grant DMS-1953815.

We are also grateful to the anonymous referees for their careful reading of the paper.

\section{Appendix: Tools from discrete Morse theory}\label{sec:DMT-Marija}

Discrete Morse theory was introduced by Forman \cite{Forman1998}, and it is a powerful tool for determining the homotopy type of a polyhedral complex. Good references  are the books of Kozlov \cite{Koz2008} and Jonsson \cite{JonssonBook2008}.  
\begin{df}[{\cite[Definition 11.1]{Koz2008}, \cite[Chapter 4, Section 2]{JonssonBook2008}}] A partial matching in a poset $P$ is a partial matching on the underlying
graph of the Hasse diagram of $P ,$ i.e., it is a subset $\mathcal{M} \subseteq P \times P$ such that
\begin{itemize}
\item $(a, b) \in \mathcal{M}$ implies $b\succ a$ (i.e., $b$ covers $a$ in $P$); 
\item each $x \in P$ belongs to at most one ordered pair in $\mathcal{M}.$ 
\end{itemize}
When $(a, b) \in \mathcal{M},$ we write $a = d(b)$ and $b = u(a).$
 A partial matching on $P$ is called \emph{acyclic} if there does not exist a cycle
 $ u(a_1)\succ a_1 \prec u(a_2) \succ a_2 \prec \dots \prec u(a_t)\succ a_t\prec u(a_1), t\ge 2.  $
with all $a_i \in P$ being distinct.  Given an acyclic partial matching $\mathcal{M}$  on $P,$ those elements of $P$ which do not belong to the matching are called \emph{unmatched} or \emph{critical} with respect to the matching $\mathcal{M}$.

\end{df}

The main result of discrete Morse theory for simplicial complexes is the following theorem.  The formulation below appears for polyhedral complexes in \cite{JJMV2019}.
\begin{theorem}[{\cite[Theorem 11.13]{Koz2008}, \cite[Theorem 4.8]{JonssonBook2008}}] \label{thm:DMT}
Let $\mathcal{K}$ be a simplicial complex and $\mathcal{M} $ be an acyclic matching on the face poset of $\mathcal{K}$. Let $c_i$ be the number of critical $i$-dimensional cells of $\mathcal{K}$ with respect to the matching $\mathcal{M}.$  Then $\mathcal{K}$ is homotopy equivalent to a cell complex $\mathcal{K}_c$ with $c_i$ cells of dimension $i\ge 0,$ plus a single $0$-dimensional cell, in the case when the empty set is also paired in the matching. 

In particular, if an acyclic matching has critical cells in only one fixed dimension $i\ge 0,$ then $\mathcal{K}$ is homotopy equivalent to a wedge of $i$-dimensional spheres.
\end{theorem}

\begin{cor}[{\cite{DeshpandeSingh2021}, \cite{JonssonBook2008}}]
If the critical cells of an acyclic matching on $\mathcal{K}$ form a subcomplex $\mathcal{K'}$ of $\mathcal{K}$, then $\mathcal{K}$ collapses simplicially to $\mathcal{K'}$, and hence $\mathcal{K}$ is homotopy equivalent to $\mathcal{K'}.$
\end{cor}

The arguments in this paper use an important special type of matching called an \emph{element matching} \cite{JonssonBook2008}.  
They are defined more generally for arbitrary collections of subsets of a set $X$ (and not just simplicial complexes) in \cite{JonssonBook2008}, and are a special case of the \emph{Pairing Lemma} in \cite[Lemma 3.4]{LinSharesh2003}.
\begin{df}[{\cite{DeshpandeSingh2021}, \cite{JonssonBook2008}}] \label{def:eltmatch}
Let $\mathcal{K}$ be a simplicial complex and $x$ a vertex.  The element matching on $\mathcal{K}$ using $x$ is  the matching 
\[ \{(\sigma, \sigma\sqcup\{x\}) \mid \sigma\sqcup\{x\}\in \mathcal{K}, x\notin \sigma\}.\]
\end{df}

The following result tells us that a sequence of element matchings is always acyclic. Again we take the explicit  formulation in \cite{DeshpandeSingh2021}.

\begin{theorem}[{\cite[Lemma 4.1]{JonssonBook2008}}] \label{thm:seq-of-elt-match}
Let $\mathcal{K}$ be a simplicial complex and $\{x_1, x_2,\dots, x_n\}$ be a subset of the vertex set of  $\mathcal{K}$. Let $\Delta_0= \mathcal{K}$ and for $i\in [n]$ define 
\begin{align*}  M(x_i)&\coloneqq\{(\sigma, \sigma\cup\{x_i\}) \mid x_i\not\in \sigma, \textrm{ and } \sigma, \sigma\cup\{x_i\}\in \Delta_{i-1}\},\\
N(x_i)&\coloneqq \{\sigma\in \Delta_{i-1} \mid \sigma\in\eta \textrm{ for some } \eta\in M(x_i)\}, \textrm{ and }\\
\Delta_i&\coloneqq\Delta_{i-1}\setminus N(x_i).
\end{align*}

Then $\bigsqcup_{i=1}^n M(x_i)$ is an acyclic matching on  $\mathcal{K}$.
\end{theorem}

\bibliographystyle{plainnat}
\bibliography{TotalCutComplexes-ArXiv-DCG-FINAL-NoComments- 2024-01-14}

\end{document}